\newcommand\bbR{\mathbb{R}}
\newcommand\bbS{\mathbb{S}}
\newcommand\bbM{\mathbb{M}}

\newcommand\cA{\mathcal{A}}
\newcommand\cC{\mathcal{C}}
\newcommand\cF{\mathcal{F}}
\newcommand\cH{\mathcal{H}}
\newcommand\cP{\mathcal{P}}
\newcommand\cT{\mathcal{T}}
\newcommand\cV{\mathcal{V}}
\newcommand\cW{\mathcal{W}}
\newcommand\cX{\mathcal{X}}

\newcommand\bn{\boldsymbol{n}}
\newcommand\bF{\boldsymbol{f}}
\newcommand\bu{\boldsymbol{u}}
\newcommand\bv{\boldsymbol{v}}
\newcommand\bw{\boldsymbol{w}}
\newcommand\bd{\boldsymbol{d}}
\newcommand\be{\boldsymbol{e}}

\newcommand\bq{\boldsymbol{q}}
\newcommand\bg{\boldsymbol{g}}
\newcommand\bI{\boldsymbol{I}}
\newcommand\beps{\boldsymbol{\varepsilon}}
\newcommand\bsig{\boldsymbol{\sigma}}
\newcommand\btau{\boldsymbol{\tau}}
\newcommand\bnabla{\boldsymbol{\nabla}}
\newcommand\bet{\boldsymbol{\eta}}
\newcommand\bchi{\boldsymbol{\chi}}
\newcommand\bpi{\boldsymbol{\pi}}

\newcommand\ubu{\underline{\boldsymbol{u}}}
\newcommand\ubv{\underline{\boldsymbol{v}}}
\newcommand\ube{\underline{\boldsymbol{e}}}
\newcommand\ubq{\underline{\boldsymbol{q}}}

\newcommand\mt{\mathtt{t}}
\newcommand\tD{\mathtt{D}}

\newcommand{\tr}{\operatorname{tr}}
\renewcommand{\div}{\operatorname{div}}
\newcommand{\bdiv}{\operatorname{\mathbf{div}}}

\documentclass[11pt, leqno]{article} 
\usepackage[utf8]{inputenc} 
\usepackage[T1]{fontenc}
\usepackage[english]{babel} 

\usepackage{lmodern} 
\usepackage{sourceserifpro} 

\usepackage{amsmath,amsthm,amsfonts,amssymb}

\usepackage{graphicx}
\usepackage{tabularray}
\usepackage{caption, subcaption}
\usepackage{tikz}

\usepackage{xfrac} 

\usepackage{stmaryrd}
\newcommand{\jump}[1]{{\llbracket{#1}\rrbracket}}
\newcommand{\mmean}[1]{\left\{\kern-1.ex\left\{ #1 \right\}\kern-1.ex\right\}}	
\newcommand{\mean}[1]{\left\{ #1 \right\}}

\usepackage{mathtools}
\DeclarePairedDelimiterX\norm[1]\lVert\rVert{
   \ifblank{#1}{\:\cdot\:}{#1}
}
\DeclarePairedDelimiterX\abs[1]\lvert\rvert{
   \ifblank{#1}{\:\cdot\:}{#1}
}
\DeclarePairedDelimiter{\inner}{(}{)}
\DeclarePairedDelimiter{\set}{\{}{\}}
\DeclarePairedDelimiter{\dual}{\langle}{\rangle}
\DeclareMathOperator*{\esssup}{ess\,sup}

\usepackage{siunitx}

\usepackage[singlespacing]{setspace} 

\usepackage{titlesec}
\titleformat{\section}[block]{\filcenter\bfseries}{\thesection.}{1em}{}
\titleformat{\subsection}[block]{\bfseries}{\thesubsection.}{1em}{}

\newtheoremstyle{paper}{}{}{\itshape}{}{\scshape}{.}{.5em}{}
\theoremstyle{paper}
\newtheorem{theorem}{Theorem}
\newtheorem{proposition}{Proposition}
\newtheorem{lemma}{Lemma}

\newtheorem{remark}{Remark}

\usepackage[numbers]{natbib}
\usepackage{natbib}
\setcitestyle{aysep={}}

\AtBeginEnvironment{thebibliography}{\setstretch{1.1}}

\usepackage{fancyhdr} 

\usepackage[pdftex, hidelinks, hypertexnames = false, hyperfootnotes = false,
pdfpagemode = UseNone, pdfdisplaydoctitle = true]{hyperref}%

\usepackage[a4paper, tmargin=3cm, bmargin=3cm, rmargin=2.2cm, lmargin=2.2cm]{geometry}

\begin{document}
\title{\huge An $hp$ Error Analysis of HDG for Linear Fluid-Structure Interaction}
\author{Salim Meddahi
\thanks{This research was  supported by Ministerio de Ciencia e Innovación, Spain, Project PID2020-116287GB-I00.}}



\maketitle

\begin{abstract}

\noindent 
A variational formulation based on velocity and stress is developed for linear fluid-structure interaction (FSI) problems. The well-posedness and energy stability of this formulation are established. To discretize the problem, a hybridizable discontinuous Galerkin method is employed. An $hp$-convergence analysis is performed for the resulting semi-discrete scheme. The temporal discretization is achieved via the Crank-Nicolson method, and the convergence properties of the fully discrete scheme are examined. Numerical experiments validate the theoretical results, confirming the effectiveness and accuracy of the proposed method.

\end{abstract}
\bigskip

\bigskip



\section{Introduction}\label{s:introduction}

Fluid-structure interaction (FSI) phenomena are prevalent across various scientific and engineering disciplines, including blood flow in arteries, aerodynamic analyses, and structural stability assessments \cite{chabannes2013, quarteroni2019, bazilevs2011, zhang2013}. Numerical solutions for FSI problems pose a significant challenge due to the inherent coupling between fluid and solid domains, nonlinearities in the governing equations, and the necessity for accurate treatment at the fluid-structure interface.

Numerical strategies for addressing FSI problems include monolithic and partitioned schemes \cite{richter2017}. Monolithic schemes \cite{bazilevs2008,gee2011, gerbeau2003, hron2006, nobile2001} solve the coupled fluid-structure system of equations simultaneously, offering improved robustness but requiring extensive computational resources \cite{heil2008, badia2008}. Partitioned approaches  solve the fluid and structure equations independently with explicit enforcement of interface conditions. These schemes can be computationally efficient but may face stability problems \cite{badia2009,banks2014, bukac2016, fernandez2013, nobile2008}.

This study adopts a monolithic approach for solving the FSI problem. To isolate the core contribution of this work and avoid complexities associated with nonlinear FSI models, we restrict our analysis to a simplified model involving  an incompressible, viscous fluid interacting with a linearly elastic solid. The governing equations are the Stokes equations for the fluid and the linear elastodynamics equations for the solid, with appropriate interface conditions to govern the interaction. 

Traditionally, FSI formulations are written in terms of the displacement in solids and the velocity pressure variables in fluids. This study introduces a  velocity-stress based formulation for linear FSI problems involving thick structures. This formulation employs global energy spaces throughout the entire domain, which simplifies treatment at the interface conditions on the discrete level. Notably, previous works \cite{selgas2018,bean2019} have also considered stress as a primary variable in FSI formulations, laying groundwork for this approach.

To manage the potential increased computational demand of this monolithic, tensorial-based variational formulation, we apply a hybridizable discontinuous Galerkin (HDG) method \cite{Cockburn2009, duSayasBook}. The hybridization technique reduces global degrees of freedom, facilitating efficient implementation through static condensation and enabling effective parallel processing, which makes it attractive for computationally demanding problems. The HDG method is also amenable to  $hp$-adaptivity and flexible mesh designs. It has shown promise in various applications such as elastodynamics and Stokes flow \cite{nguyen2011, cockburnFu2018, duSayas2020, nguyenPeraire2010, cockburnSayas2014}, and its adaptation to the FSI velocity-pressure-displacement formulation has seen recent developments \cite{sheldon2016,Fu2022}. 

The presented HDG discretization method is specifically designed for the velocity-stress formulation of the linear FSI problem, utilizing symmetric tensors with piecewise polynomial entries of arbitrary degree $k \geq 0$ to approximate each stress component in both 2D and 3D. Additionally, the discrete velocity field and the discrete trace variable defined on the mesh skeleton are piecewise polynomials of degree $k+1$. The stability and convergence of the semi-discrete scheme are proven, alongside obtaining $hp$ error estimates for stress and velocity in the corresponding $L^2$-norms. Convergence rates are shown to be quasi-optimal with respect to mesh size, while only suboptimal by half a power concerning polynomial degree. Moreover, the fully discrete scheme based on the Crank-Nicolson method is demonstrated to be stable and convergent.

This paper is structured as follows. We first complement this section with notations and definitions of functional spaces. Section~\ref{sec:model} introduces the linear FSI model problem and its corresponding stress/velocity weak formulation.  Existence and uniqueness of a solution are established in Section~\ref{sec:wellposedness}.  Section~\ref{sec:FE} reviews essential $hp$-requirements for our analysis. Section~\ref{sec:semi-discrete} introduces the semi-discrete HDG method and proves its well-posedness.  Section~\ref{sec:convergence} details the $hp$ convergence analysis of the HDG method.  The fully discrete scheme is addressed in Section~\ref{sec:fully-discrete}. Finally, numerical results validating expected convergence rates in both two and three dimensions are presented in Section~\ref{sec:numresults}.

\subsection{Notations and Sobolev spaces}

Let $\mathbb{M}$ represent the space of real matrices with dimensions $d\times d$. We denote by $\bbS:= \set{\btau\in \bbM;\ \btau = \btau^{\mt}}$ the subspace of symmetric matrices, where $\btau^{\mt}:=(\tau_{ji})$ is the transpose of $\btau = (\tau_{ij})$. The component-wise inner product of two matrices $\bsig, \,\btau \in\bbM$ is given by $\bsig:\btau:= \sum_{i,j}\sigma_{ij}\tau_{ij}$. Additionally, we introduce the deviatoric part $\btau^{\tD}:=\btau-\frac{1}{d}\left(\tr\btau\right) \bI$ of a tensor $\btau$, where $\tr\btau:=\sum_{i=1}^d\tau_{ii}$ and $\bI$ represents the identity in $\bbM$.

Let $D$ be a polyhedral Lipschitz bounded domain of $\bbR^d$ $(d=2,3)$, with boundary $\partial D$. Throughout this work, we apply all differential operators row-wise.  Hence, given a tensorial function $\bsig:D\to \bbM$ and a vector field $\bu:D\to \bbR^d$, we set the divergence $\bdiv \bsig:D \to \bbR^d$, the   gradient $\bnabla \bu:D \to \bbM$, and the linearized strain tensor $\beps(\bu) : D \to \bbS$ as
\[
(\bdiv \bsig)_i := \sum_j   \partial_j \sigma_{ij} \,, \quad (\bnabla \bu)_{ij} := \partial_j u_i\,,
\quad\hbox{and}\quad \beps(\bu) := \frac{1}{2}\left[\bnabla\bu+(\bnabla\bu)^{\mt}\right].
\]
 For $s\in \bbR$, $H^s(D, E)$ stands for the usual Hilbertian Sobolev space of functions with domain $D$ and values in $E$, where $E$ is either $\bbR$, $\bbR^d$ or $\bbS$. In the case $E=\bbR$ we simply write $H^s(D)$. The norm of $H^s(D,E)$ is denoted $\norm{\cdot}_{s,D}$ and the corresponding semi-norm $|\cdot|_{s,D}$, indistinctly for $E=\bbR,\bbR^d,\bbS$. We use the convention  $H^0(D, E):=L^2(D,E)$ and let $(\cdot, \cdot)_D$ be the inner product in $L^2(D, E)$, for $E\in \set{\bbR,\bbR^d,\bbS}$; namely,
\begin{equation*}
	(\bu, \bv)_D:=\int_D \bu\cdot\bv,\ \forall \bu,\bv\in L^2(D,\bbR^d),\quad  (\bsig, \btau)_D:=\int_D \bsig:\btau,\ \forall \bsig, \btau\in L^2(D,\bbS).
\end{equation*}
The space of tensors in $L^2(D, \bbS)$ with divergence in $L^2(D, \bbR^d)$ is denoted $H(\bdiv, D, \bbS)$. The corresponding norm is given by $\norm{\btau}^2_{H(\div,D)}:=\norm{\btau}_{0,D}^2+\norm{\bdiv\btau}^2_{0,D}$. Let $\bn$ be the outward unit normal vector to $\partial D$, the Green formula
\begin{equation}\label{green}
	(\btau, \beps(\bv))_D + (\bdiv \btau, \bv)_D = \int_{\partial D} \btau\bn\cdot \bv\qquad  \forall \bv \in H^1(D,\bbR^d),
\end{equation}
can be used to extend the normal trace operator $\btau \to (\btau|_{\partial D})\bn$ to a linear continuous mapping $(\cdot|_{\partial D})\bn:\, H(\bdiv, D, \bbS) \to H^{-\frac{1}{2}}(\partial D, \bbR^d)$, where $H^{-\frac{1}{2}}(\partial D, \bbR^d)$ is the dual of $H^{\frac{1}{2}}(\partial D, \bbR^d)$.

Since we are dealing with an evolution problem, besides the Sobolev spaces defined above, we need to introduce function spaces defined over a bounded time interval $(0, T)$ and taking values in a separable Banach space $V$, whose norm is denoted $\norm{\cdot}_{V}$.  For $1 \leq p \leq\infty$, $L^p_{[0,T]}(V)$ represents the space of B\"ochner-measurable functions $f: (0,T) \to V$ with finite norms, given by
\[
\norm{f}^p_{L^p_{[0,T]}(V)}:= \int_0^T\norm{f(t)}_{V}^p\, \text{d}t \quad\hbox{for $1\leq p < \infty$ and} \quad \norm{f}_{L_{[0,T]}^\infty(V)}:= \esssup_{[0, T]} \norm{f(t)}_V.
\]

We denote the Banach space of continuous functions $f: [0,T] \to V$ by $\mathcal{C}^0_{[0,T]}(V)$. Moreover, for any $k \in \mathbb{N}$, $\mathcal{C}^k_{[0,T]}(V)$ represents the subspace of $\mathcal{C}^0_{[0,T]}(V)$ consisting of functions $f$ with strong derivatives $\frac{d^j f}{dt^j}$ in $\mathcal{C}^0_{[0,T]}(V)$ for all $1 \leq j \leq k$. In what follows, we will interchangeably use the notations $\dot{f}:= \frac{d f}{dt}$, $\ddot{f}:= \frac{d^2 f}{dt^2} $, and $\dddot{f}:= \frac{d^3 f}{dt^3} $ to denote the first, second, and third derivatives with respect to $t$. 

Finally, we introduce the Sobolev space
\[
\begin{array}{c}
W_{[0,T]}^{1, p}(V):= \left\{f: \ \exists g\in L_{[0,T]}^p( V)
\ \text{and}\ \exists f_0\in V\ \text{such that} \ f(t) = f_0 + \int_0^t g(s)\, \text{d}s\quad \forall t\in [0,T]\right\},
\end{array}
\]
and denote $H^1_{[0,T]}(V):= W_{[0,T]}^{1,2}(V)$. The space $W_{[0,T]}^{k, p}(V)$ is defined recursively  for all $k\in\mathbb{N}$.

Throughout the rest of this paper, we shall use the letter $C$  to denote generic positive constants independent of the mesh size  $h$ and the polynomial degree $k$ and the time step $\Delta t$. These constants may stand for different values at different occurrences. Moreover, given any positive expressions $X$ and $Y$ depending on $h$ and $k$, the notation $X \,\lesssim\, Y$  means that $X \,\le\, C\, Y$.

\section{Velocity/stress formulation of a linear FSI problem}\label{sec:model}

This section presents a simplified model for a fluid-structure interaction problem, focusing on a global stress tensor description. We neglect convective terms in the fluid and assume small deformations, enabling the consideration of fully linear models within time-independent fluid and solid subdomains. We consider a polygonal/polyhedral domain $\Omega\subset \bbR^d$, $d=2,3$, divided into two polygonal/polyhedral subdomains $\Omega_f$ and $\Omega_s$ representing the fluid and solid regions respectively. In this way, the global domain $\Omega = \Omega_f \cup \Sigma\cup \Omega_s$ comprises three disjoint components, where $\Sigma:= \overline{\Omega}_f \cap  \overline{\Omega}_s$ is the interface connecting the fluid to the structure.

We assume that $\Omega_s$ represents an elastic body with mass density $\rho_s>0$, governed by the equation of motion:
\begin{equation}\label{motion}
	\rho_s\ddot{\bd}-\bdiv\bsig_s =\bF_s \quad \text{in $\Omega_s\times (0, T]$},
\end{equation}
where $\bF_s:\Omega_s\times[0, T] \to \bbR^d$ is a body force acting over the time interval $[0,T]$. The vector field  $\bd:\Omega\times[0, T] \to \bbR^d$ is the displacement and $\bsig_s:\Omega_s\times[0,T]\to \bbS$ is the stress tensor. The linearized strain tensor $\beps(\bd)$ is related to the stress through Hooke's law
\begin{equation}\label{Constitutive:solid}
	\bsig_s  = \cC_s \beps(\bd)  \quad \text{in $\Omega_s\times (0, T]$},
\end{equation}
where $\cC_s$ is a symmetric and positive definite constant tensor of order 4.

We consider the partition $\partial \Omega_s = \Gamma_s \cup \Sigma$ of the solid boundary. To simplify the exposition, we assume that the elastic body is clamped ($\bd = \mathbf 0$) at $\Gamma_s  \times (0, T]$,  where the boundary subset $\Gamma_s$ is of positive surface measure.  Finally, we impose the initial conditions 
\begin{equation}\label{inits}
	\bd(0) = \bd^0 \quad\text{in $\Omega_s$} \quad \text{and} \quad 
 \dot{\bd}(0) = \bd^1 \quad\text{in $\Omega_s$}.
\end{equation}

In the subdomain $\Omega_f$ a fluid with mass density $\rho_f>0$ and velocity field $\bu_f:\Omega_f \to \bbR^d$ is governed by the time-dependent Stokes equations, 
\begin{equation}\label{Stokes}
	\begin{split}
		\rho_{f}\dot{\bu}_f - \bdiv\bsig_f &=\bF_f \quad \text{in $\Omega_f\times (0, T]$},
		\\
		\div \bu_f &= 0 \quad \text{in $\Omega_f\times (0, T]$},
	\end{split}
\end{equation}
where $\bF_f:\Omega_f\to \bbR^d$ is the body force and the stress tensor $\bsig_{\!f}$ is expressed in terms $\bu_f$ and the pressure $p:\Omega_f\to \bbR$ as
\begin{equation}\label{Constitutive:fluid}
	\bsig_f = 2\mu_f \beps(\bu_f) - p \bI \quad \text{in $\Omega_f\times (0, T]$}.
\end{equation}
Here, $\mu_f>0$ is the dynamic viscosity of the fluid. We partition the fluid boundary $\partial \Omega_f$ into $\Gamma_f \cup \Sigma$  and  enforce the no-slip boundary condition $\mathbf{u}_f = \mathbf{0}$ on $\Gamma_f \times (0, T]$. To complete the Stokes system, we incorporate the initial condition:
\begin{equation}\label{initf}
	\bu_f(0) = \bu_f^0 \quad\text{in $\Omega_f$}.
\end{equation}

In what follows, we consider a penalty formulation for the Stokes problem \eqref{Stokes}-\eqref{Constitutive:fluid}. Namely, the incompressibility constraint is enforced by penalization through the equation $\div \bu_f = -\frac{1}{\lambda_f} p$, with a sufficiently large coefficient $\lambda_f>0$ that plays the role of a penalty parameter. This amounts to consider the constitutive relation $\bsig_f = \cC_f \beps(\bu_f)$, where $\cC_f\beps(\bu_f)  := 2\mu_f \beps(\bu_f) + \lambda_f\tr \beps(\bu_f) \bI$. In what follows, for brevity, we continue to denote the solutions obtained through penalization as $\bu_f$ and $p$.

The formulation of the fluid-structure interaction problem involves coupling the fluid and solid problems through two transmission conditions enforced on the shared interface $\Sigma$. The first condition, known as the kinematic condition, ensures the continuity of velocities between the fluid and solid particles across the interface. It is expressed as $\bu_f = \dot\bd$ on $\Sigma\times (0, T]$. The second condition, referred to as the dynamic condition, imposes the continuity of normal components of stress tensors between the fluid and solid across the interface. This condition is mathematically represented as $\bsig_f\bn_f + \bsig_s\bn_s = \mathbf 0$ on $\Sigma\times (0, T]$, where $\bn_s$ and $\bn_f$ stand for the unit outward normal vectors to the boundaries $\partial\Omega_s$ and $\partial\Omega_f$, respectively.
 
Our objective is to treat the fluid and solid stress tensors as primary unknowns.  Introducing the notation $\cA_s:= \cC_s^{-1}$, the elasticity problem can be formulated as  
 \begin{align}\label{stressFormulation:solid}
 \begin{split}
 \rho_s\dot{\bu}_s-\bdiv \bsig_s &=\bF_s  \quad\text{in $\Omega_s\times (0, T]$}, 
 \\
 \cA_s \dot{\bsig}_s &= \beps(\bu_s)  \quad \text{in $\Omega_s\times (0, T]$}, 
 \\
 \bu_s &= \mathbf 0
 \quad  \text{on $\Gamma_s\times (0, T]$},
\end{split}
\end{align}
in terms of the velocity $\bu_s = \dot\bd$, and $\bsig_s$. 

In the fluid domain, we can combine \eqref{Constitutive:fluid} and the approximate incompressibility condition to eliminate the pressure $p$ and end up with the boundary value problem
\begin{align}\label{stressFormulation:fluid}
	\begin{split}
	\rho_f \dot\bu_f -  \bdiv \bsig_f &= \bF_f \quad \text{in $\Omega_f\times (0, T]$},
	\\
		\cA_f \bsig_f &=  \beps(\bu_f)\quad \quad \text{in $\Omega_f\times (0, T]$},
	\\
	\bu_f &= \mathbf{0} \quad \quad \text{on $\Gamma_f\times (0, T]$},
	\end{split}	
\end{align}
where $\cA_f:= \cC_f^{-1}$ is given by $\cA_f\btau  = \frac{1}{2\mu_f} \left( \btau - \frac{\lambda_f}{d \lambda_f + 2 \mu_f} \tr \btau \bI\right)$. The transmission conditions  are expressed as
\begin{align}\label{transmission}
	\begin{split}
		\bsig_s\bn_s + \bsig_f\bn_f &= \mathbf 0\quad \text{on $\Sigma\times (0, T]$}, 
	\\
	\bu_s &= \bu_f \quad \text{in $\Sigma\times (0, T]$}.
	\end{split}
\end{align}

We introduce a global stress variable $\bsig:\Omega \to \bbS$, which is defined in terms of the elastic and fluid components of the stress as  $\bsig|_{\Omega_s} :=\bsig_s$ and $\bsig|_{\Omega_f}:= \bsig_f$, respectively. We extend this convention to arbitrary elements $\btau \in L^2(\Omega,\bbS)$  and let $\btau_s := \btau|_{\Omega_s}$ and $\btau_f:= \btau|_{\Omega_f}$. The subscripts $s$ and $f$ will be omitted when the context is clear. 

We introduce the space $\cH := L^2(\Omega,\bbS)$ equipped with the inner product 
\[
	\inner{\bsig, \btau}_\cH := \inner{\cA_s\bsig, \btau}_{\Omega_s} + \inner{ \cA_f \bsig, \btau}_{\Omega_f},
\]
and denote $\norm{\btau}_{\cH}^2:= \norm{\cA_s^{\sfrac12}\btau}^2_{0,\Omega_s} + \norm{\cA_f^{\sfrac12}\btau}^2_{0,\Omega_f}$ the corresponding norm, where $\norm{\cA_s^{\sfrac12}\btau}^2_{0,\Omega_s}:= \inner{\cA_s\btau, \btau}_{\Omega_s}$ and $\norm{\cA_f^{\sfrac12}\btau}^2_{0,\Omega_f}:=\inner{ \cA_f \btau, \btau}_{\Omega_f}$.  We point out that, according to our assumption on the tensor $\cC_s$,   there exist  constants $a_s^+>0$ and $a_s^->0$ such that
\begin{equation}\label{normHs}
	a_s^+ \norm{\btau}^2_{0, \Omega_s}  \leq \norm{\cA_s^{\sfrac12}\btau}^2_{0,\Omega_s} \leq a_s^- \norm{\btau}^2_{0, \Omega_s}  \quad \forall \btau \in L^2(\Omega_s,\bbS).
\end{equation}
Moreover, taking into account that 
\(
\inner{\cA_f \bsig,\btau}_{\Omega_{f}} = \frac{1}{2\mu_f} \inner{\bsig^\tD,\btau^\tD}_{\Omega_{f}} + \frac{1}{d \lambda_f + 2 \mu_f} \inner{\tr \bsig, \tr \btau}_{\Omega_{f}},
\)
and using the identity $\norm{\btau}^2_{0,\Omega_f} = \norm{\btau^\tD}^2_{0,\Omega_f} + \tfrac{1}{d} \norm{\tr \btau}^2_{0,\Omega_f}$ we obtain the bounds 
\begin{equation}\label{normHf}
	a_f^- \norm{\btau}^2_{0, \Omega_f}  \leq \norm{\cA_f^{\sfrac12}\btau}^2_{0,\Omega_f} \leq a_f^+ \norm{\btau}^2_{0, \Omega_f}  \quad \forall \btau \in L^2(\Omega_f, \bbS),
\end{equation}
with $a_f^+:= \frac{d}{2\mu_f}$ and $a_f^-:= \max\{ \frac{1}{2\mu_f}, \frac{d}{d\lambda_f + 2\mu_f} \}$.

Owing to the dynamic transmission condition, the global stress $\bsig(t)\colon \Omega \to \bbS$ belongs to the energy space $H(\bdiv,\Omega, \bbS)$. Alongside the tensorial variable $\bsig$, we define $\bu$ as the global velocity field, with $\bu|_{\Omega_s}:= \bu_s$ and $\bu|_{\Omega_f}:= \bu_f$. Additionally, we introduce the coefficient $\rho:\Omega\to \mathbb{R}$ and the function $\bF:\Omega\times [0, T] \to \mathbb{R}^d$, where their restrictions to $\Omega_s$ and $\Omega_f$ coincide with the mass density and the body force of each respective domain. With these definitions at hand, we deduce from the first equations of \eqref{stressFormulation:solid} and \eqref{stressFormulation:fluid} that
\begin{equation}\label{global1}
	\rho\dot{\bu} =  \bdiv \bsig  + \bF  \quad\text{in $\Omega\times (0, T]$}.
\end{equation}

Let us now consider an arbitrary element $\btau\in H(\bdiv,\Omega, \bbS)$. Testing the second equation in \eqref{stressFormulation:solid} with $\btau_s=\btau|_{\Omega_s}$ and using Green's formula \eqref{green} we  obtain
\begin{equation}\label{var:solid}
	\inner{\cA_s\dot\bsig, \btau}_{\Omega_s}  = \inner*{\beps(\dot{\bu}_s), \btau }_{\Omega_s} 
		= - \inner*{\bu_s, \bdiv\btau  }_{\Omega_s} + \dual*{\bu_s, \btau\bn_s}_{\partial\Omega_s},
\end{equation}
where $\dual{\cdot, \cdot}_{\partial\Omega_s}$ stands for the duality pairing between $H^{-\frac{1}{2}}(\partial\Omega_s, \bbR^d)$ and $H^{\frac{1}{2}}(\partial\Omega_s, \bbR^d)$.
Next, testing the second equation in \eqref{stressFormulation:fluid} with $\btau_f= \btau|_{\Omega_f}$ and integrating by parts yield 
\begin{equation}\label{var:fluid}
	\inner{\cA_f\bsig, \btau}_{\Omega_f} = \inner{\beps(\bu_f), \btau }_{\Omega_f} = - \inner*{\bu_f, \bdiv \btau}_{\Omega_f}+ \dual{\bu_f, \btau\bn_f}_{\partial\Omega_f}.
\end{equation}
Incorporating the boundary conditions on $\Gamma_s$ and $\Gamma_f$, along with the transmission condition $\mathbf{u}_s = \mathbf{u}_f$ on $\Sigma$, and combining equations \eqref{var:solid} and \eqref{var:fluid}, yields
\begin{equation}\label{global2}
	\inner{\cA_s\dot\bsig, \btau}_{\Omega_s} + \inner{\cA_f\bsig, \btau}_{\Omega_f}  + \inner{\bu , \bdiv\btau }_\Omega = 0,\quad \forall \btau\in H(\bdiv,\Omega, \bbS).
\end{equation}

We introduce the function $\bu^0\colon \Omega \to \bbR^d$ defined by $\bu^0|_{\Omega_s} := \bd^1$ and $\bu^0|_{\Omega_f} := \bu_f^0$. We also let $\bsig_s^0:=\cC_s\beps(\bd^0)$, $\bsig_f^0:=\cC_f\beps(\bu_f^0)$, and $\bsig_s^1:=\cC_s\beps(\bd^1)$ and introduce the global tensor $\bsig^0\colon \Omega \to \bbS$ given by  $\bsig^0|_{\Omega_s} = \bsig_s^0$ and $\bsig^0|_{\Omega_f} = \bsig_f^0$. In what follows, we will need the following regularity assumptions on the data:
\begin{align}\label{dataregularity}
	\begin{split}
		\bF &\in W^{1,1}_{[0,T]}(L^2(\Omega, \bbR^d)), \quad  \bu^0 \in L^2(\Omega, \bbR^d), \quad \bsig_s^1 \in L^2(\Omega_s,\bbS),
		\\
		 &  \quad\text{and}\quad   \bsig^0 \in H(\bdiv,\Omega,\bbS)\cap H^r(\Omega_s\cup \Omega_f,\bbS), \quad   r> 1/2.
	\end{split}
\end{align}

We deduce from \eqref{global1} and \eqref{global2} the following  velocity/stress variational formulation of the fluid-solid problem: find  $\bu\in W^{1,\infty}_{[0,T]}(L^2(\Omega,\bbR^d))\cap L^2_{[0,T]}(H^1_0(\Omega,\bbR^d))$ and $\bsig \in H^1_{[0,T]}(\cH) \cap L^2_{[0,T]} ( H(\bdiv,\Omega, \bbS))$ such that 
\begin{align}\label{var:golbal}
	\begin{split}
		\inner{\rho\dot\bu, \bv}_{\Omega} - \inner{\bdiv \bsig, \bv}_{\Omega} &= \inner{\bF, \bv}_{\Omega}  \quad \forall \bv\in L^2(\Omega, \bbR^d),
		\\
		\inner{\cA_s\dot\bsig, \btau}_{\Omega_s} + \inner{\cA_f\bsig, \btau}_{\Omega_f}  + \inner{\bu , \bdiv\btau }_\Omega &= 0,\quad \forall \btau\in H(\bdiv,\Omega, \bbS),
	\end{split}
\end{align}
with the initial conditions 
\begin{equation}\label{init0}
	\bu(0) = \bu^0\quad \text{and} \quad \bsig_s(0)=  \bsig_s^0. 
\end{equation}

\begin{remark}
	For simplicity, we have employed homogeneous Dirichlet boundary conditions on all boundaries. It is important to acknowledge that practical applications often involve non-homogeneous or Neumann boundary conditions for the fluid and solid boundary domains. These more general scenarios are considered in the numerical results presented in Section~\ref{sec:numresults}.
\end{remark}
\section{Well-posedness of the  weak formulation}\label{sec:wellposedness}

Let $\set{\cT^c_h}$ denote a family of conforming and shape-regular partitions of $\overline\Omega$ into tetrahedra if $d = 3$ and into triangles if $d = 2$. We assume that the partitions $\cT^c_h$ are aligned with the interface $\Sigma$ separating the fluid and solid domains. For all $K\in \cT_h^c$, we let $\cP_m(K, E)$ be the space of polynomials of degree at most $m \geq 0$ on $K$ with values in $E\in \set{\bbS, \bbR^d}$. The couple of finite-dimensional spaces   
\[
\cX^{AW}_K := \set{\btau \in \cP_{d+1}(K,\bbS);\ \bdiv \btau \in \cP_1(T,\bbR^d)} \quad \text{and} \quad \cP_{1}(K,\bbR^d)
\]
represent the lowest order finite element introduced in \cite{Arnold2002, Arnold2008} for the mixed formulation of the elasticity system. This finite element method is only needed as a theoretical tool to establish the well-posedness of the fluid-solid stress formulation \eqref{var:golbal}-\eqref{init0}. Detailed information regarding the degrees of freedom associated with $\cW^{AW}_K$ and the construction of the corresponding interpolation operator $\Pi^{AW}_K$ can be found in  \cite{Arnold2002, Arnold2008}. We define the global finite element spaces 
\begin{align*}
	\cX^{AW}_h &:= \set{\btau \in H(\bdiv, \Omega, \bbS);\ \btau|_K \in \cW^{AW}_K \quad \forall K\in \cT^c_h}
	\\
	\cP_1(\cT^c_h, \bbR^d) &:= \set{\bv \in L^2(\Omega, \bbR^d);\ \bv|_K \in \cP_1(K, \bbR^d) \quad \forall K\in \cT^c_h}.
\end{align*}
It is shown in \cite{Arnold2002, Arnold2008} that the global interpolation operator $\Pi^{AW}_h\colon H(\bdiv,\Omega,\bbS)\cap H^{r}(\Omega,\bbS)\to \cX^{AW}_h$, defined for $r> 1/2$ by $(\Pi^{AW}_h\btau)|_{K} := \Pi^{AW}_K (\btau|_K)$, $\forall K\in \cT^c_h$, satisfies the commutativity property  
\begin{equation}\label{commPi}
	\bdiv \Pi^{AW}_h\btau = Q_h \bdiv \btau \quad \forall \btau\in H(\bdiv,\Omega,\bbS)\cap H^{r}(\Omega,\bbS),\quad r> 1/2,
\end{equation}
and the stability property 
\begin{equation}\label{stabPi}
	\norm{\Pi^{AW}_h \btau }_{0,\Omega} \leq C \norm{\btau}_{r,\Omega} \quad \forall \btau\in H(\bdiv,\Omega,\bbS)\cap H^{r}(\Omega,\bbS),\quad r> 1/2,
\end{equation}
with $C>0$ independent of $h$. Here, $Q_h\colon L^2(\Omega,\bbR^d)\to \cP_1(\cT_h^c, \bbR^d)$ represents the $L^2$-orthogonal projection. Additionally, for all $\btau\in H(\bdiv,\Omega,\bbS)\cap H^{r}(\Omega,\bbS)$, it holds that
\[
	\lim_{h\to 0} \norm{\btau - \Pi^{AW}_h\btau}_{H(\bdiv,\Omega)} = 0.
\]
This approximation property, combined with the density of smooth functions in $H(\bdiv,\Omega,\bbS)$, implies that the set $\set{\cX^{AW}_h}_h$ is dense in $H(\bdiv,\Omega,\bbS)$. Furthermore, it is well-known that $\set{\cP_1(\cT_h^c, \bbR^d}_h$ is also dense in $L^2(\Omega,\bbR^d)$.

We consider the problem of finding $\bu_h\in \cC^1(\cP_1(\cT_h^c, \bbR^d))$ and $\bsig_h\in \cC^1_{[0,T]}(\cX^{AW}_h)$ such that
\begin{equation}\label{var:golbal_h}
	\begin{split}
		\inner{\rho\dot\bu_h, \bv}_{\Omega} - \inner{\bdiv \bsig_h, \bv}_{\Omega} &= \inner{\bF, \bv}_{\Omega}  \quad \forall \bv\in \cP_1(\cT_h^c, \bbR^d),
		\\
		\inner{\cA_s\dot\bsig_h, \btau}_{\Omega_s} + \inner{\cA_f\bsig_h, \btau}_{\Omega_f}  + \inner{\bu_h , \bdiv\btau }_\Omega &= 0,\quad \forall \btau\in \cX^{AW}_h,
	\end{split}
\end{equation}
with the initial conditions 
\begin{equation}\label{init0_h}
	\bu_h(0) = Q_h\bu^0\quad \text{and} \quad \bsig_{h}|_{\Omega_s}(0)= (\Pi^{AW}_h \bsig^0)|_{\Omega_s}.
\end{equation}

\begin{proposition}
	Problem \eqref{var:golbal_h}-\eqref{init0_h} has a unique solution.
\end{proposition}
\begin{proof}
	Given $\btau_h\in \cX^{AW}_h$, let $\bet \in H(\bdiv, \Omega, \bbS)$ be defined by $\bet|_{\Omega_s} := \btau_h|_{\Omega_s}$ and $\bet|_{\Omega_f} := \beps (\bw_f)$ where $\bw_f\in H^1(\Omega_f,\bbR^d)$ is the unique solution of the elliptic problem
	\begin{align}
		\begin{split}
			\bdiv \beps(\bw_f)   &= \mathbf 0\quad \text{in $\Omega_f$},
			\\
			\bw_f &= \mathbf 0 \quad \text{on $\Gamma_f$},
			\\
			\beps(\bw_f)\bn_f &= \btau_h\bn_f \quad \text{on $\Sigma$}.
		\end{split}
	\end{align}
	The problem above is well-posed. Moreover, as the Neumann data $\btau_h\bn_f$ on $\Sigma$  belongs to $H^\delta(\Sigma,\bbR^d)$ for all $0 \leq \delta <1/2$,  classical regularity results \cite{Grisvard1986,Dauge1988} for the elasticity equations in polygonal/polyhedral domains ensure the existence of $r_f\in (1/2, 1)$, which depends on the geometry of $\Omega_f$, such that $\bw_f \in H^{1+r_f}(\Omega_f,\bbR^d)$. It follows that the operator $E_h\colon \cX^{AW}_h \to \cX^{AW}_h$ defined by $E_h \btau_h = \Pi_h^{AW} \bet$
	is well-defined. Moreover, the $\cX^{AW}_K$-invariance of $\Pi_K^{AW}$ ensures that  $E_h\btau_h|_{\Omega_s} = \btau_h|_{\Omega_s}$ and the diagram commuting property \eqref{commPi} implies that   $\bdiv (E_h\btau_h)|_{\Omega_f} = \mathbf 0$. The linear operator $E_h$ induces the direct splitting $\cX^{AW}_h = \cX_{0,h}^{AW} \oplus \cX^{AW}_{\bot, h}$ where
	\[
		\cX_{0,h}^{AW} := (I - E_h)\cX^{AW}_h
		\quad \text{and}
		\quad
		\cX^{AW}_{\bot, h} := E_h(\cX^{AW}_h).
	\]
	The elements $\btau_h\in \cX^{AW}_h$ can then be uniquely split as $\btau_h = \btau_h^0 + \btau_h^\bot$, with $\btau_h^0:= \btau_h - E_h\btau_h\in \cX^{AW}_{0, h}$ and $\btau_h^\bot:= E_h \btau_h \in \cX^{AW}_{\bot, h}$.  It follows that problem \eqref{var:golbal_h} can be equivalently rewritten as follows: find $\bu_h\in \cC^1(\cP_1(\cT_h^c, \bbR^d))$, $\bsig_h^\bot\in \cC^1(\cX^{AW}_{\bot, h})$ and $\bsig_h^0\in \cC^1(\cX_{0,h}^{AW})$ such that $\bsig_h(t) = \bsig_h^\bot + \bsig_h^0$ and
	\begin{equation}\label{splitvar:golbal_h}
		\begin{split}
			\inner{\rho\dot\bu_h, \bv}_{\Omega} - \inner{\bdiv \bsig_h, \bv}_{\Omega} &= \inner{\bF, \bv}_{\Omega}  \quad \forall \bv\in \cP_1(\cT_h^c, \bbR^d),
			\\
			\inner{\cA_s\dot\bsig_h, \btau_h^\bot}_{\Omega_s} + \inner{\cA_f\bsig_h, \btau_h^\bot}_{\Omega_f}  + \inner{\bu_h , \bdiv\btau_h^\bot }_\Omega &= 0,\quad \forall \btau_h^\bot\in \cX^{AW}_{\bot, h}
			\\
			\inner{\cA_s\dot\bsig_h, \btau_h^0}_{\Omega_s} + \inner{\cA_f\bsig_h, \btau_h^0}_{\Omega_f}  + \inner{\bu_h , \bdiv\btau_h^0 }_\Omega &= 0,\quad \forall \btau_h^0\in \cX_{0,h}^{AW}
		\end{split}
	\end{equation}
	
	We observe that $\btau\mapsto \norm{\cA_f^{\sfrac12}\btau^\bot}_{0,\Omega_f}$ is a norm on $\cX_{0,h}^{AW}$, while $\btau\mapsto \norm{\cA_s^{\sfrac12}\btau^\bot}_{0,\Omega_s}$ induces a norm on $\cX_h^\bot$. The latter assertion is justified by  construction, as if $\btau\in \cX^{AW}_{\bot, h}$ satisfies $\btau|_{\Omega_s} = \mathbf 0$, then it follows that $\btau = \mathbf 0$ in $\Omega$.
	
	Keeping in mind that $\btau_h^0|_{\Omega_s} = \mathbf 0$ for all $\btau_h^0\in \cX_{0,h}^{AW}$, we deduce from the last equation of \eqref{splitvar:golbal_h} that
	\begin{equation}\label{var:subgolbal_h0}
		\inner{\cA_f\bsig_h, \btau_h^0}_{\Omega_f}  + \inner{\bu_h , \bdiv\btau_h^0 }_\Omega = 0,\quad \forall \btau_h^0\in \cX_{0,h}^{AW}.
	\end{equation}
	This induces us to consider the operator $L_h: \cP_1(\cT_h^c, \bbR^d)\times \cX^{AW}_{\bot, h} \to \cX_{0,h}^{AW}$, where $L_h(\bv_h,\btau_h^\bot) := \btau_h^0$ is the unique solution of the linear system of equations
	\begin{equation*}
		\inner{\cA_f\tau_h^0, \bet_h^0}_{\Omega_f} =   - \inner{\bv_h , \bdiv\bet_h^0 }_{\Omega_f} - \inner{\cA_f\btau_h^\bot, \bet_h^0}_{\Omega_f},\quad \forall \bet_h^0\in \cX_{0,h}^{AW}.
	\end{equation*}
	From \eqref{var:subgolbal_h0}, it is clear that the $\bsig_h^0$ component of $\bsig_h$ is determined by $\bsig_h^0 = L_h(\bu_h, \bsig_h^\bot)$. Consequently, according to \eqref{splitvar:golbal_h}, the solution components $\bu_h\in \mathcal{P}_1(\mathcal{T}h^c, \mathbb{R}^d)$ and $\bsig_h^\bot\in \mathcal{X}^{AW}_{\bot, h}$ of problem \eqref{splitvar:golbal_h} satisfy the first-order system of ordinary differential equations
	\begin{align}\label{var:golbal_h2}
		\begin{split}
			\inner{\rho\dot\bu_h, \bv}_{\Omega} &=  \inner{\bdiv \bsig_h, \bv}_{\Omega} + \inner{\bF, \bv}_{\Omega}  \quad \forall \bv\in \cP_1(\cT_h^c, \bbR^d),
			\\
			\inner{\cA_s\dot\bsig_h^\bot, \btau_h^\bot}_{\Omega_s} &= - \inner{\cA_f\bsig_h^\bot, \btau_h^\bot}_{\Omega_f}  - \inner{\bu_h , \bdiv\btau_h^\bot}_{\Omega_s}  - \inner{\cA_f L_h(\bu_h, \bsig_h^\bot), \btau_h^\bot}_{\Omega_f},\quad \forall \btau_h^\bot\in \cX^{AW}_{\bot, h},
		\end{split}
	\end{align}	
	with the initial conditions $\bu_h(0) = Q_h\bu^0$ and $\bsig_h^\bot|_{\Omega_s}(0
	) = (\Pi^{AW}_h \bsig^0)|_{\Omega_s}$.
	Using that the bilinear form $  \inner{\rho\bu, \bv}_{\Omega} + \inner{\cA_s\bsig, \btau}_{\Omega_s}$ induces a norm on $\cP_1(\cT_h^c, \bbR^d)\times \cX^{AW}_{\bot, h}$, we deduce from the standard theory for constant coefficient systems of linear  differential equations that \eqref{var:golbal_h2} has a unique solution of class $\cC^1$ on $[0, T]$.  This result immediately yields the existence of the Galerkin approximate solutions $(\bu_h(t),\bsig_h(t))$, with $\bsig_h(t) = \bsig_h^\bot(t) + L_h(\bu_h(t), \bsig_h^\bot(t))$,  satisfying \eqref{var:golbal_h}-\eqref{init0_h}.
\end{proof}
 We proceed to derive a priori estimates for the set $\set{(\bu_h(t),\bsig_h(t))}_h$.
\begin{lemma}
	Under conditions \eqref{dataregularity}, there exists a constant $C>0$ independent of $h$ such that
	\begin{align}\label{apriori}
		\begin{split}
			&\norm{\bu_h}^2_{W^{1,\infty}_{[0,T]}(L^2(\Omega,\bbR^d))} + \norm{\bsig_h}_{H^{1}_{[0,T]}(\cH)}^2 
			+ \norm{\bdiv\bsig_h}_{L^2_{[0,T]}(L^2(\Omega,\bbR^d)) }^2 
			\\
			&\qquad \leq C \left( \norm{\bF}_{W^{1,1}_{[0,T]}(L^2(\Omega, \bbR^d))}^2 + \norm{\bu^0}_{0,\Omega} +  \norm{\bsig^0}_{H(\div,\Omega)} + \norm{\bsig_s^1}_{0,\Omega_s}\right).
		\end{split}
	\end{align}
\end{lemma}
\begin{proof}
	Taking $\bv = \bu_h$ and $\btau = \bsig_h$ in \eqref{var:golbal_h} and adding the resulting equations, we obtain
	\begin{equation}\label{energy}
		\frac{1}{2}\frac{\text{d}}{\text{d}t}\norm{\rho^{\sfrac12}\bu_h(t)}^2_{0,\Omega} + \frac{1}{2}\frac{\text{d}}{\text{d}t}\norm{\cA_s^{\sfrac12}\bsig_h(t)}^2_{0,\Omega_s} +  \norm{\cA_f^{\sfrac12}\bsig_h(t)}^2_{0,\Omega_f} = \inner{\bF, \bu_h}_{\Omega}.
	\end{equation}	
	Integrating in time we obtain
	\begin{equation}\label{energy1}
		\begin{split}
			&\frac{1}{2}\max_{[0,T]}\norm{\rho^{\sfrac12}\bu_h(t)}^2_{0,\Omega} + \frac{1}{2}\max_{[0,T]}\norm{\cA_s^{\sfrac12}\bsig_h(t)}^2_{0,\Omega_s} +  \int_{0}^T \norm{\cA_f^{\sfrac12}\bsig_h(t)}^2_{0,\Omega_f} \text{d}t 
			\\
			&\qquad = \int_{0}^T \inner{\bF(t), \bu_h}_{\Omega}\ \text{d}t+ \frac{1}{2}\norm{\rho^{\sfrac12}\bu_h(0)}^2_{0,\Omega} + \frac{1}{2}\norm{\cA_s^{\sfrac12}\bsig_h(0)}^2_{0,\Omega_s}.
		\end{split}
	\end{equation}
	Applying the Cauchy-Schwartz inequality to the first term on the right-hand side and taking into account \eqref{normHs} and the initial conditions \eqref{init0_h}, we deduce that  
	\begin{equation}\label{energy2}
		\begin{split}
			&\frac{\rho^-}{2}\max_{[0,T]}\norm{\bu_h(t)}^2_{0,\Omega} + \frac{1}{2}\max_{[0,T]}\norm{\cA_s^{\sfrac12}\bsig_h(t)}^2_{0,\Omega_s} +  \int_{0}^T \norm{\cA_f^{\sfrac12}\bsig_h(t)}^2_{0,\Omega_f} \text{d}t 
			\\
			&\qquad \leq  \norm{\bF}_{L^1_{[0,T]}(L^2(\Omega,\bbR^d))} \max_{[0,T]}\norm{\bu_h(t)}_{0,\Omega}+ \frac{\rho^+}{2}\norm{Q_h\bu^0}^2_{0,\Omega} + \frac{a^+_s}{2}\norm{\Pi^{AW}_h\bsig^0}^2_{0,\Omega_s},
		\end{split}
	\end{equation}
	where $\rho^+ = \max\{\rho_s, \rho_f\}$ and $\rho^- = \min\{\rho_s, \rho_f\}$. Using the stability of the projector $Q_h$ and \eqref{stabPi} we readily deduce from Young's inequality  the estimate 
	\begin{align}\label{energyEstimate}
		\begin{split}
			&\max_{[0,T]}\norm{\bu_h(t)}^2_{0,\Omega} + \frac{1}{2}\max_{[0,T]}\norm{\cA_s^{\sfrac12}\bsig_h(t)}^2_{0,\Omega_s} +  \int_{0}^T \norm{\cA_f^{\sfrac12}\bsig_h(t)}^2_{0,\Omega_f} \text{d}t 
			\\
			&\qquad \qquad \leq C_0 \left( \norm{\bF}^2_{L^1_{[0,T]}(L^2(\Omega,\bbR^d))} + \norm{\bu^0}^2_{0,\Omega} + \norm{\bsig^0}^2_{s,\Omega_s}\right),
		\end{split}
	\end{align}
	with a constant $C_0>0$ independent of $h$. 
	
	To establish stability estimates for the divergence term, we rely on the regularity assumptions specified in \eqref{dataregularity}. We recall that the  Sobolev embedding $W^{1,1}_{[0,1]}(L^2(\Omega,\bbR^d)) \hookrightarrow \cC^0_{[0,T]}(L^2(\Omega,\bbR^d))$ holds true, as established in \cite[Lemma 7.1]{roubivcek2013}. In other words, there exists a constant $C_1>0$  such that
	\begin{equation}\label{embedding}
		\max_{[0,T]}\norm{\bg(t)}_{L^2(\Omega,\bbR^d)} \leq C_1 \norm{\bg}_{W^{1,1}_{[0,T]}(L^2(\Omega,\bbR^d))}\quad \forall \bg \in W^{1,1}_{[0,T]}(L^2(\Omega,\bbR^d)).
	\end{equation}
	
	Taking the derivative of the second equation in \eqref{var:golbal_h} yields 
	\begin{equation}\label{deriv}
		\inner{\cA_s\ddot\bsig_h, \btau}_{\Omega_s} + \inner{\cA_f\dot\bsig_h, \btau}_{\Omega_f}  + \inner{\dot\bu_h , \bdiv\btau }_\Omega = 0,\quad \forall \btau\in \cX^{AW}_h.
	\end{equation}
	Next, we notice that the first equation in \eqref{var:golbal_h} can be written 
	\begin{equation}\label{udot}
		\dot\bu_h =   Q_h(\tfrac{1}{\rho}\bdiv \bsig_h + \tfrac{1}{\rho} \bF).
	\end{equation}
	Substituting back this expression for $\dot\bu_h$ in \eqref{deriv} and using that $\bdiv (\cX^{AW}_h) \subset \cP_1(\cT_h, \bbR^d)$ yield
	\begin{equation}\label{2dt}
		\inner{\cA_s\ddot\bsig_h, \btau}_{\Omega_s} + \inner{\cA_f\dot\bsig_h, \btau}_{\Omega_f} + \inner{\tfrac{1}{\rho}\bdiv\bsig_h, \bdiv\btau_h}_{\Omega} = - \inner{\tfrac{1}{\rho}\bF, \bdiv\btau_h}_\Omega \quad \forall \btau \in \cX^{AW}_h.
	\end{equation}
	We assign to this equation the initial conditions 
	\begin{equation}\label{init0_h2}
		\bsig_h(0)=  \Pi^{AW}_h \bsig^0,\quad \text{and} \quad \dot\bsig_{s,h}(0)=  X^s_h \bsig_s^1,
	\end{equation}
	where  $X^s_h\colon L^2(\Omega_s,\bbS) \to \set{\btau|_{\Omega_s},\ \btau \in \cX^{AW}_h}$ is the $L^2$-orthogonal projection.
	
	Taking $\btau = \dot\bsig_h$ in \eqref{2dt} gives 
	\[
		\frac{1}{2}\frac{\text{d}}{\text{d}t}\norm{\cA_s^{\sfrac12}\dot\bsig_h}^2_{0,\Omega_s} + \norm{\cA_f^{\sfrac12}\dot\bsig_h}^2_{0,\Omega_f} + \frac{1}{2}\frac{\text{d}}{\text{d}t}\norm{\tfrac{1}{\rho^{\sfrac12}}\bdiv\bsig_h}^2_{0,\Omega} = - \inner{\tfrac{1}{\rho}\bF, \bdiv\dot\bsig_h}_\Omega.
	\]
Integrating this identity over $[0, t]$, $t\leq T$,  and applying integration by parts on the right-hand side, we obtain
\begin{align*}
	\begin{split}
			&\frac{1}{2}\norm{\cA_s^{\sfrac12}\dot\bsig_h(t)}^2_{0,\Omega_s} + \frac{1}{2} \int_0^t \norm{\cA_f^{\sfrac12}\dot\bsig_h(s)}^2_{0,\Omega_f} \text{d}s + \frac{1}{2}\norm{\tfrac{1}{\rho^{\sfrac12}}\bdiv\bsig_h(t)}^2_{\Omega} 
			\\
			& \qquad =  \int_0^t \inner{\tfrac{1}{\rho}\dot\bF(s), \bdiv\bsig_h}_\Omega \text{d}s + \inner{\tfrac{1}{\rho}\bF(0), \bdiv\bsig_h(0)}_\Omega - \inner{\tfrac{1}{\rho}\bF(t), \bdiv\bsig_h}_\Omega
			\\
			&\qquad \qquad + \frac{1}{2}\norm{\cA_s^{\sfrac12}\dot\bsig_h(0)}^2_{0,\Omega_s} + \frac{1}{2}\norm{\tfrac{1}{\rho^{\sfrac12}}\bdiv\bsig_h(0)}^2_{\Omega}.
	\end{split}
\end{align*}
Using the Cauchy-Schwartz inequality to the inner products involving  the source term and taking into account \eqref{normHs}, the initial conditions \eqref{init0_h2}, and the Sobolev embedding \eqref{embedding}, we deduce that  
\begin{align*}
	\begin{split}
		&\frac{1}{2}\max_{[0,T]}\norm{\cA_s^{\sfrac12}\dot\bsig_h(t)}^2_{0,\Omega_s} + \frac{1}{2} \int_0^T \norm{\cA_f^{\sfrac12}\dot\bsig_h(t)}^2_{0,\Omega_f} \text{d}t + \frac{1}{2\rho^+}\max_{[0,T]}\norm{\bdiv\bsig_h(t)}^2_{0,\Omega} 
			\\
			& \qquad  \leq  \frac{\max\set{1,C_0}}{\rho^-} \norm{\dot\bF}_{W^{1,1}_{[0,T]}(L^2(\Omega,\bbR^d))} \left( \max_{[0,T]}\norm{\bdiv\bsig_h(t)}_{0,\Omega} + \norm{Q_h\bdiv\bsig^0}_{0,\Omega} \right) 
			\\
			&\qquad \qquad + \frac{a_s^+}{2}\norm{X_h^s\bsig^1}^2_{0,\Omega_s} + \frac{1}{2\rho^-}\norm{Q_h\bdiv\bsig^0}^2_{0,\Omega}.
	\end{split}
\end{align*}
The stability of the $L^2$-projection operators $X_h^s$ and $Q_h$ and a straightforward application of Young's inequality that there exists a constant $C_3>0$ independent of $h$ such that
	\begin{align}\label{energyEstimate2}
		\begin{split}
			&\max_{[0,T]}\norm{\cA_s^{\sfrac12}\dot\bsig_h(t)}^2_{0,\Omega_s} +  \int_0^T \norm{\cA_f^{\sfrac12}\dot\bsig_h(t)}^2_{0,\Omega_f} \text{d}t + \max_{[0,T]}\norm{\bdiv\bsig_h(t)}^2_{0,\Omega} 
			\\
			& \qquad \leq C_3  \left( \norm{\bF}_{W^{1,1}_{[0,T]}(L^2(\Omega, \bbR^2))}^2 + \norm{\bdiv \bsig^0}^2_{0,\Omega} +  \norm{\bsig_s^1}^2_{0,\Omega_s}\right).
		\end{split}
	\end{align}
	Finally, it follows from \eqref{udot} that 
	\begin{equation}\label{energyEstimate3}
		\max_{[0,T]} \norm{\dot\bu_h}^2_{0,\Omega} \leq C_4 \left( \max_{[0,T]} \norm{\bF}_{0,\Omega}^2 + \max_{[0,T]} \norm{\bdiv\bsig_h}_{0,\Omega}^2\right), 
	\end{equation}
	with $C_4>0$ independent of $h$. Summing up, it follows from  \eqref{energyEstimate},  \eqref{energyEstimate2}, and  \eqref{energyEstimate3} that 
	\begin{align*}
		\begin{split}
			&\norm{\bu_h}^2_{W^{1,\infty}_{[0,T]}(L^2(\Omega,\bbR^d))} +   \norm{\bsig_h}_{H^1_{[0,T]}(\cH)}^2
			+ \norm{\bdiv\bsig_h}_{L^\infty_{[0,T]}(L^2(\Omega, \bbR^d)) }^2 
			\\ 
			&\qquad \leq C_5 \left(  \norm{\bF}_{W^{1, 1}_{[0,T]}(L^2(\Omega, \bbR^d))}^2 + 
			\norm{\bu^0}^2_{0,\Omega} + \norm{\bsig^0}^2_{s,\Omega_s} + \norm{\bdiv \bsig^0}^2_{0,\Omega} +  \norm{\bsig_s^1}^2_{0,\Omega_s}
			\right),
		\end{split}
	\end{align*}	
	with $C_5>0$ independent of $h$, and the result follows.
\end{proof}

We will now pass to the limit in the Galerkin approximations \eqref{var:golbal_h}-\eqref{init0_h} to prove the existence of a solution $(\bu, \bsig)$ to problem \eqref{var:golbal}-\eqref{init0}. 

\begin{theorem}
	Assume that $\bF$, $\bu^0$, $\bsig^0$, and $\bsig_s^1$ satisfy conditions \eqref{dataregularity}. Then, there exists a unique solution $(\bu, \bsig)$ of problem \eqref{var:golbal}-\eqref{init0} satisfying
	\begin{align}\label{aprioriCont}
		\begin{split}
			&\norm{\bu}^2_{W^{1,\infty}_{[0,T]}(L^2(\Omega,\bbR^d))} + \norm{\bsig}_{H^{1}_{[0,T]}(\cH)}^2 
			+ \norm{\bdiv \bsig}_{L^2_{[0,T]}(L^2(\Omega, \bbR^d)) }^2 
			\\
			&\qquad \leq C \left( \norm{\bF}_{W^{1,1}_{[0,T]}(L^2(\Omega, \bbR^d))}^2 + \norm{\bu^0}_{0,\Omega} +  \norm{\bsig^0}_{H(\div,\Omega)} + \norm{\bsig_s^1}_{0,\Omega_s}\right),
		\end{split}
	\end{align}
	with a constant $C>0$ independent of $h$.
\end{theorem}
\begin{proof}
Let $\set{(\bu_h, \bsig_h)}_h$ be the sequence of Galerkin approximations obtained as a solution of \eqref{var:golbal_h}-\eqref{init0_h}. By virtue of the energy estimate \eqref{apriori}, we can extract a subsequence of $\set{(\bu_h, \bsig_h)}_h$, still denoted by $\set{(\bu_h, \bsig_h)}_h$, such that the weak star$^*$ and weak  convergences 
\begin{align*}
	\begin{split}
		&\bu_h \overset{*}{\rightharpoonup} \bu \quad \text{in $W^{1,\infty}_{[0,T]}(L^2(\Omega,\bbR^d))$ \qquad  $\bsig_h \rightharpoonup \bsig$ in $H^{1}_{[0,T]}(\cH)$},
		\\
		&\text{and} \qquad \bdiv \bsig_h \rightharpoonup \bdiv \bsig \quad \text{in $L^2_{[0,T]}(L^2(\Omega,\bbR^d))$},
	\end{split}
\end{align*} 
hold true. Passing to the limit in \eqref{apriori} we deduce that the limit pair $(\bu, \bsig)$ satisfies \eqref{aprioriCont}. In addition, by virtue of the Sobolev embeddings $W^{1,\infty}_{[0,T]}(L^2(\Omega,\bbR^d)) \hookrightarrow \cC^0_{[0,T]}(L^2(\Omega,\bbR^d))$ and $H^{1}_{[0,T]}(L^2(\Omega,\bbS)) \hookrightarrow \cC^0_{[0,T]}(L^2(\Omega,\bbS))$ (see \cite[Chapter XVIII]{dautray2000}), we have that $(\bu_h(0), \bsig_h|_{\Omega_s}(0))$ also converges weakly to $(\bu(0), \bsig|_{\Omega_s}(0))$ in $L^2(\Omega,\bbR^d) \times L^2(\Omega,\bbS)$. Moreover, since $\bu_h(0)= Q_h\bu^0$ converges to $\bu^0$ in $L^2(\Omega,\bbR^d)$ and $\bsig_h(0)|_{\Omega_s} = (\Pi^{AW}_h \bsig^0)|_{\Omega_s})$ converges to $\bsig^0_s$ in $L^2(\Omega_s,\bbS)$, we deduce that the initial conditions $\bu(0)= \bu^0$ and $\bsig(0)|_{\Omega_s} = \bsig^0_s$ are verified. 

To prove that the limit pair $(\bu, \bsig)$ satisfies \eqref{var:golbal}, we use the usual limiting process. We fix $h_0>0$ and test \eqref{var:golbal_h} with $\varphi(t)\bv$ and $\psi(t)\btau$, where $\bv \in \cP_1(\cT_{h_0}^c, \bbR^d)$,  $\btau\in \cX^{AW}_{h_0}$, and  $\varphi, \psi \in \cC^0([0,T])$. Integrating over $[0, T]$ we deduce that, that for each $h\leq h_0$
\begin{align*}
	\begin{split}
		\int_{0}^{T} \Big[ \inner{\rho\dot\bu_h(t), \bv}_{\Omega} - \inner{\bdiv \bsig_h(t), \bv}_{\Omega} - \inner{\bF(t), \bv}_{\Omega} \Big]\varphi(t)\text{d}t  &= 0,	
		\\
		\int_{0}^{T} \Big[ \inner{\cA_s\dot\bsig_h(t), \btau}_{\Omega_s} + \inner{\cA_f\bsig_h(t), \btau}_{\Omega_f} +  \inner{\bu_h(t) , \bdiv\btau }_\Omega  \Big]\psi(t)\text{d}t & = 0.
	\end{split}
\end{align*}
By passing to the limit as $h\to 0$, we find
\begin{align*}
	\begin{split}
		\int_{0}^{T} \Big[ \inner{\rho\dot\bu(t), \bv}_{\Omega} - \inner{\bdiv \bsig(t), \bv}_{\Omega} - \inner{\bF(t), \bv}_{\Omega} \Big]\varphi(t)\text{d}t  &= 0,	
		\\
		\int_{0}^{T} \Big[ \inner{\cA_s\dot\bsig(t), \btau}_{\Omega_s} + \inner{\cA_f\bsig(t), \btau}_{\Omega_f} +  \inner{\bu(t) , \bdiv\btau }_\Omega  \Big]\psi(t)\text{d}t & = 0.
	\end{split}
\end{align*}
The density of $\set{\cP_1(\cT_{h}^c, \bbR^d)}_h$ in $L^2(\Omega,\bbR^d)$ and $\set{\cX^{AW}_{h}}_h$ in $H(\bdiv,\Omega,\bbS)$ imply that $(\bu, \bsig)$ satisfies \eqref{var:golbal}.

Let us show now that $\bu \in L^2_{[0,T]}(H^1_0(\Omega,\bbR^d))$. To this end,  we test the second equation in \eqref{var:golbal} with an indefinitely differentiable function $\btau$ having compact support contained in $\Omega$, then integrate by parts to deduce that $\beps(\bu)|_{\Omega_s} = \cA_s \dot\bsig_s(t)$ and $\beps(\bu)|_{\Omega_f} = \cA_f \bsig_f(t)$, which implies that $\beps(\bu) \in L^2_{[0,T]}(\cH)$. Using Korn's inequality, we further conclude that $\bu\in L^2_{[0,T]}(H^1(\Omega,\bbR^d))$. Next, applying Green's formula \eqref{green} to the second equation in \eqref{var:golbal} proves that  $\bu\in L^2_{[0,T]}(H_0^1(\Omega,\bbR^d))$.

To establish uniqueness, let us consider a solution $(\bu, \bsig)$  to \eqref{var:golbal}-\eqref{init0} with homogeneous data, namely $\bF = \mathbf 0$, $\bu^0 = \mathbf 0$, and $\bsig^0_s = \mathbf 0$. By applying to the system \eqref{var:golbal} the same steps that led to \eqref{energy1} we immediately deduce that $\bu(t) = \mathbf 0$ and $\bsig(t) = \mathbf 0$ for all $t\in [0, T]$, and the uniqueness follows. 
\end{proof}
 
\begin{remark}
	The $H(\bdiv)$-conforming and symmetry-preserving mixed finite elements discussed in this section, initially introduced by Arnold and Winther \cite{Arnold2002} in 2D and further developed by Arnold, Awanou, and Winther \cite{Arnold2008} in 3D, encounter computational complexity issues due to the significant number of required degrees of freedom. In Section~\ref{sec:FE}, a more efficient and easier-to-implement alternative for high-order approximations is proposed.	
\end{remark}

\begin{remark}
	We emphasize that equation \eqref{2dt} offers a pure-stress formulation of the FSI problem. However, in this study, we chose to utilize the velocity/stress formulation \eqref{var:golbal}-\eqref{init0} instead, as it aligns better with the requirements of the HDG discretization method. 
\end{remark}

\section{Finite element spaces and approximation properties}\label{sec:FE}

This section revisits well-established $hp$ approximation properties, adapting them to our specific notation and future requirements.

Let $\cT_h$ be a shape regular triangulation of the domain $\bar \Omega$ into tetrahedra and/or parallelepipeds if $d=3$ and into triangles and/or quadrilaterals if $d=2$. We allow $\cT_h$ to have hanging nodes and assume that it is aligned with the interface $\Sigma$ between $\Omega_s$ and $\Omega_f$. Specifically, $\bar{\Omega}_s$ and $\bar{\Omega}_f$ coincide with the unions of the elements of the sets $\cT_h^s := \set{K\in \cT_h;\ K\subset \bar{\Omega}_s}$ and $\cT_h^f := \set{K\in \cT_h;\ K\subset \bar{\Omega}_f}$, respectively.  As a consequence, $\cA_s$ and $\cA_f$ are constant tensors and $\rho$ is a constant function within each element of $\cT_h$. We denote by $h_K$ the diameter of $K$ and let the parameter $h:= \max_{K\in \cT_h} \{h_K\}$ represent the size of the mesh $\cT_h$.

We define a closed subset $F\subset \overline{\Omega}$ to be an interior edge/face if it has a positive $(d-1)$-dimensional measure and can be expressed as the intersection of the closures of two distinct elements $K$ and $K'$, i.e., $F =\bar K\cap \bar K'$. On the other hand, a closed subset $F\subset \overline{\Omega}$ is a boundary edge/face if there exists $K\in \cT_h$ such that $F$ is an edge/face of $K$ and $F =  \bar K\cap \partial \Omega$. We consider the set $\cF_h^0$ of interior edges/faces and the set $\cF_h^\partial$ of boundary edges/faces and let $\cF_h = \cF_h^0\cup \cF_h^\partial$. We denote by $h_F$ the diameter of an edge/face $F\in\cF_h$ and assume that $\cT_h$ is locally quasi-uniform with constant $\gamma>0$. This means that, for all $h$ and all $K\in \cT_h$, we have that   
\begin{equation}\label{reguT}
	  h_F \leq h_K\leq \gamma h_F\quad \forall F\in \cF(K), 
\end{equation}
where $\cF(K)$ represents the set  of edges/faces composing the element $K\in \cT_h$. This condition implies that neighboring elements have similar sizes.  

For all $s\geq 0$, the broken Sobolev space with respect to the partition $\cT_h$ of $\bar \Omega$ is defined as
\[
 H^s(\cT_h,E):=
 \set{\bv \in L^2(\Omega, E): \quad \bv|_K\in H^s(K, E)\quad \forall K\in \cT_h},\quad \text{for $E \in \set{ \bbR, \bbR^d, \bbS}$}. 
\]
Following the convention mentioned earlier, we write $H^0(\cT_h,E) = L^2(\cT_h,E)$ and $H^s(\cT_h,\bbR) = H^s(\cT_h)$. We introduce the inner product 
\[
 \inner{\psi, \varphi}_{\cT_h} := \sum_{K\in \cT_h} \inner{\psi, \varphi}_{ K}\quad \forall  \psi, \varphi \in L^2(\cT_h, E),\quad E\in \set{\bbR, \bbR^d, \bbS}
\]
and write $\norm{\psi}^2_{0,\cT_h}:= \inner{\psi, \psi}_{\cT_h}$. Accordingly, we let $\partial \cT_h:=\set{\partial K;\ K\in \cT_h}$ be the set of all element boundaries and define $L^2(\partial \cT_h,\bbR^d)$ as the space of vector-valued functions which are square-integrable on each $\partial K\in \partial \cT_h$. We define 
\[
  \dual{\bu, \bv}_{\partial \cT_h} := \sum_{K\in \cT_h} \int_{\partial K} \bu\cdot \bv, 
  \quad \text{and} \quad  
  \norm{\bv}^2_{0, \partial \cT_h}:= \dual{\bv, \bv}_{\partial \cT_h}
  \quad
  \forall  \bu, \bv\in L^2(\partial \cT_h,\bbR^d),
\]
where  $\dual{\bu, \bv}_{\partial K} := \sum_{F\in \cF(K)} \int_F\bu\cdot \bv$.  Besides, we equip the space $L^2(\cF_h,\bbR^d):= \prod_{F\in \mathcal{F}_h} L^2(F, \bbR^d)$ with the inner product  
\[
(\bu, \bv)_{\cF_h} := \sum_{F\in \cF_h} \int_F\bu\cdot \bv \quad \forall \bu,\bv\in L^2(\cF_h,\bbR^d),
\]
and denote the corresponding norm $\norm*{\bv}^2_{0,\cF_h}:= (\bv,\bv)_{\cF_h}$. 

Hereafter,  $\cP_m(D)$ is the space of polynomials of degree at most $m\geq 0$ on $D$ if $D$ is a triangle/tetrahedron, and the space of polynomials of degree at most $m$ in each variable in $D$ if $D$ is a quadrilateral/parallelepiped.  The space of $E$-valued functions with components in $\cP_m(D)$ is denoted $\cP_m(D, E)$ where $E$ is either $\bbR^d$,  or $\bbS$. We introduce the space of piecewise-polynomial functions   
\[
 \cP_m(\cT_h) := 
 \set{ v\in L^2(\cT_h): \ v|_K \in \cP_m(K),\ \forall K\in \cT_h }
 \]
with respect to the partition $\cT_h$ and  the space of piecewise-polynomial functions 
 \[
 \cP_m(\cF_h) := 
 \set{ \hat\bv\in L^2(\cF_h): \ \hat\bv|_F \in \cP_m(F),\ \forall F\in \cF_h }
 \]
 with respect to the partition $\cF_h$. The subspace of $L^2(\cT_h, E)$ with components in $\cP_m(\cT_h)$ is denoted $\cP_m(\cT_h, E)$ for $E\in \set{\bbR^d, \bbS}$. Likewise, $\cP_m(\cF_h, \bbR^d)$ stands for the subspace of $L^2(\cF_h, \bbR^d)$ with components in $ \cP_m(\cF_h)$.  We finally consider 
  \[
  \cP_m(\partial \cT_h, \bbR^d) := \set*{\bv\in L^2(\partial \cT_h, \bbR^d);\ \bv|_{\partial K}\in  \cP_m(\partial K, \bbR^d),\ \forall K\in \cT_h},
\]
where  $\cP_m(\partial K, \bbR^d):= \prod_{F\in \cF(K)} \cP_m(F, \bbR^d)$. It is important to keep in mind that, by definition,  the functions in $L^2(\partial \cT_h,\bbR^d)$ and $\cP_m(\partial \cT_h, \bbR^d)$ are multi-valued on every interior face $F$, whereas the functions in $L^2(\cF_h,\bbR^d)$ and $\cP_m(\cF_h, \bbR^d)$ are single-valued on each face $F$.

We consider  $\bn\in \cP_0(\partial \cT_h, \bbR^d)$, where $\bn|_{\partial K}=\bn_K $ is the unit normal vector to $\partial K$ oriented towards the exterior of $K$. Obviously, if $F = K\cap K'$ is an interior edge/face of $\cF_h$, then $\bn_K = -\bn_{K'}$ on $F$. If  $\bv\in H^s(\cT_h,\bbR^d)$ and $\btau\in H^s(\cT_h,\bbS)$, with $s>\sfrac{1}{2}$, the functions $\bv|_{\partial \cT_h}\in L^2(\partial \cT_h,\bbR^d)$ and $(\btau|_{\partial \cT_h}) \bn\in L^2(\partial \cT_h,\bbR^d)$ are meaningful by virtue of the trace theorem. For the same reason, if $\bv \in H^1(\Omega, \bbR^d)$ then $\hat\bv:= \bv|_{\cF_h}$ is well-defined in $L^2(\cF_h, \bbR^d)$.

For $k\geq 0$, we introduce the finite dimensional subspace of $\mathcal H$ given by 
\[
  \mathcal H_{h} := \cP_{k}(\cT_h,\bbS),
\]
and consider the following discrete trace inequality. 

\begin{lemma}\label{TraceDG}
There exists a constant $C>0$ independent of $h$ and $k$ such that 	
\begin{equation}\label{discTrace}
  \norm{\tfrac{h^{\sfrac{1}{2}}_{\cF}}{k+1} \btau\bn}_{0,\partial \cT_h}\leq C \norm*{ \btau}_{\cH}\quad \forall  \btau\in \cH_h. 
 \end{equation}
\end{lemma}
\begin{proof} See \cite[Lemma 3.2]{meddahi2023hp}. 
\end{proof}

For any integer $m\geq 0$ and $K\in \cT_h$, we denote by $\Pi_K^m$ the $L^2(K)$-orthogonal projection onto $\cP_{m}(K)$. The global projection $\Pi^m_\cT$ in $L^2(\cT_h)$ onto $\cP_m(\cT_h)$ is then given by $(\Pi^m_\cT v)|_K = \Pi_K^m(v|_K)$ for all $K\in \cT_h$. Similarly, the global projection $\Pi_\cF^m$ in $L^2(\cF_h)$ onto $\cP_{m}(\cF_h)$ is given,  separately for all $F\in \cF_h$,  by $(\Pi^m_\cF \hat v)|_{F} = \Pi_F^m(\hat v|_F)$, where   $\Pi^m_F$ is the $L^2(F)$-orthogonal projection onto $\cP_m(F)$. In the following, we maintain the notation $\Pi_\cT^m$ to refer to the $L^2$-orthogonal projection onto $\mathcal{P}_m(\mathcal{T}_h,E)$, for $E\in {\mathbb{R}^d, \mathbb{S}}$. It is noteworthy that the tensorial version of $\Pi_\cT^m$ inherently preserves the symmetry of matrices, as it is derived by applying the scalar operator componentwise. Similarly, we will use $\Pi_\cF^m$ to denote either the $L^2$-orthogonal projection onto $\mathcal{P}_m(\mathcal{F}_h)$ or $\mathcal{P}_m(\mathcal{F}_h, \mathbb{R}^d)$.

In the remainder of this section, we provide approximation properties for the  projectors defined above. A detailed proof of these results can be found in \cite[Section 3]{meddahi2023hp} and the references therein.

\begin{lemma}\label{maintool}
There exists a constant $C>0$ independent of $h$ and $k$  such that 
\begin{equation}\label{tool1}
	\norm{\btau - \Pi^k_\cT \btau}_{\cH} + \norm{ \tfrac{h^{\sfrac{1}{2}}_{\cF}}{k+1}(\btau - \Pi_\cT^k\btau)\bn}_{0,\partial \cT_h} 
	 \leq 
	C \tfrac{h_K^{\min\{ r, k \}+1}}{(k+1)^{r+1}} \Big( \norm*{\btau_s}^2_{1+r,\Omega_s}+ \norm*{\btau_f}^2_{1+r,\Omega_s}  \Big)^{\sfrac{1}{2}},
\end{equation}
	for all $\btau \in \mathcal{H}$ such that $\btau\in  H^{1+r}(\Omega_s\cup\Omega_f,\bbS)$, with $r\geq 0$.
\end{lemma}
\begin{proof}
	The result is a straightforward adaptation of \cite[Lemma 3.3]{meddahi2023hp}.
\end{proof}

In what follows, given $\ubv := (\bv, \hat \bv ) \in H^s( \cT_h,\bbR^d)\times L^2( \cF_h,\bbR^d)$, with $s>\sfrac{1}{2}$, we define $\jump{\ubv} \in L^2(\partial \cT_h,\bbR^d)$ by  
\[
\jump{\ubv}|_{\partial K} := \bv|_{\partial K} - \hat \bv|_{\partial K}\quad  \forall K\in \cT_h.
\] 
We introduce the space $\underline{\cV} := H^1_0(\cT_h,\bbR^d)\times L^2( \cF^0_h,\bbR^d)$, where 
\[
  L^2(\cF^0_h,\bbR^d):=\set{\bv\in L^2(\cF_h,\bbR^d);\ \bv|_F = \mathbf 0,\ \forall F\in \cF_h^\partial},
\]
and endowed it with the semi-norm   
\begin{equation}\label{norm:sym}
  \norm{\ubv}^2_{\underline{\cV}} = \norm{\nabla \bv}^2_{0,\cT_h} + \norm{\tfrac{k+1}{h_\cF^{\sfrac{1}{2}}}\jump{\ubv}}^2_{0, \partial \cT_h}  \quad \forall \ubv:= (\bv, \hat\bv)\in \underline{\cV},
\end{equation}
where $h_\cF\in \cP_0(\cF_h)$ is given by $h_\cF|_F := h_F$ for all $F \in \cF_h$.
For $k\geq 0$, we consider the finite-dimensional subspace $\underline{\cV}_{h} := \cP_{k+1}(\cT_h,\bbR^d) \times \cP_{k+1}(\cF^0_h, \bbR^d)$ of $\underline{\cV}$, where  
\[ 
\cP_{k+1}(\cF^0_h, \bbR^d) := \set{\hat \bv \in \cP_{k+1}(\cF_h, \bbR^d);\ \hat \bv|_F = \mathbf 0,\ \forall F\in \cF_h^\partial}.
\] 
We introduce the operator 
$\Xi_h^{k+1}:\, \underline{\cV} \to \underline{\cV}_h$ given by $\Xi_h^{k+1}\ubv = (\Pi_\cT^{k+1} \bv, \Pi^{k+1}_\cF \hat \bv)$, for all $\ubv =(\bv, \hat \bv)\in \underline{\cV}$. Moreover, we consider  the linear operator $\underline i:\, H^1_0(\Omega,\bbR^d) \to \underline{\cV} $ defined, for any $\bv \in H^1_0(\Omega, \bbR^d)$, by 
\[
  \underline i(\bv) :=(\bv, \hat\bv)\in \underline{\cV}, \quad \text{with}\ \hat\bv := \bv|_{\cF_h}.
\]
We point out that $\jump{\underline i(\bv)} = \mathbf 0$ for all $\bv \in H^1_0(\Omega,\bbR^d)$. 

\begin{lemma}\label{maintool2}
	There exists a constant $C>0$ independent of $h$ and $k$  such that 
	\begin{equation}\label{tool2}
	\norm{ \underline i(\bu) - \Xi_h^{k+1}\underline i(\bu)}_{\underline{\cV}}  
		 \leq 
		C \tfrac{h^{\min\{ r, k \}+1}}{(k+1)^{r+\sfrac{1}{2}}} \Big( \sum_{j=1}^J \norm{\bu}^2_{2+r,\Omega_j}\Big)^{\sfrac{1}{2}},
	\end{equation}
		for all  $\bu \in H_0^1(\Omega,\bbR^d)\cap H^{2+r}(\cup_j\Omega_j,\bbR^d)$ , $r \geq 0$.
	\end{lemma}
	\begin{proof}
		See \cite[Lemma 3.4]{meddahi2023hp}. 
	\end{proof}
	
	\section{Introduction of the HDG semi-discrete method}\label{sec:semi-discrete}

	We propose the following HDG space-discretization method for problem \eqref{var:golbal}-\eqref{init0}:  find  $\bsig_h \in  \cC^1_{[0,T]}(\mathcal{H}_h)$ and  $\ubu_h := (\bu_h, \hat\bu_h) \in \cC^1_{[0,T]}(\underline{\cV}_h)$  satisfying  
	\begin{align}\label{sd}
	\begin{split}
		\inner{\rho \dot\bu_h, \bv}_{\cT_h} &+ \inner{\cA_s\dot{\bsig}_{h}, \btau }_{\Omega_s}  +  \inner*{ \cA_f\bsig_h, \btau}_{\Omega_f}  + B_h(\bsig_h, \ubv) - B_h(\btau, \ubu_h)
		\\
		& +\dual{ \tfrac{(k+1)^2}{h_\cF}\jump{\ubu_h}, \jump{\ubv}}_{\partial \cT_h} 
	= \inner*{\bF(t), \bv}_{\cT_h}\quad \forall \btau \in \cH_h,\ \forall \ubv = (\bv, \hat \bv)\in \underline{\cV}_h,
	\end{split}
	\end{align}
	where the bilinear form $B_h(\cdot,\cdot)$ is defined by:
	\[
	  B_h(\btau, \ubv) := \inner*{\btau, \beps(\bv)}_{\cT_h} - \dual*{ \btau\bn,  \jump{\ubv}}_{\partial \cT_h}.  
	\]  
	We start up problem \eqref{sd} with the initial conditions 
	\begin{equation}\label{initial-R1-R2-h*c}
		\ubu_h(0) = \Xi^{k+1}_\cT \underline i (\bu^0), \quad \bsig_{h}|_{\Omega_s}(0)= \Pi_\cT^k \bsig^0_s.
	\end{equation}
	
	We have the following boundedness property for the bilinear form $B_h(\cdot,\cdot)$.
	\begin{proposition}
		There exists a constant $C>0$ independent of $h$ and $k$ such that  
	\begin{equation}\label{Bhh}
		|B_h(\btau, \ubv)| \leq C  \norm{\btau}_\cH  \norm{ \ubv}_{\underline{\cV}}\quad \text{for all}\ \btau \in \mathcal H_h \ \text{and} \ \ubv = \underline{\cV}.
	\end{equation}
	\end{proposition}
	\begin{proof}
		Applying the Cauchy-Schwarz inequality, \eqref{normHs}, and \eqref{normHf}, we deduce that 
	\begin{equation}\label{Bh}
		  |B_h(\btau, \ubv)| \lesssim ( \norm{\btau}^2_\cH + \norm{\tfrac{h_\cF^{\sfrac{1}{2}}}{k+1} \btau\bn}^2_{0,\partial \cT_h})^{\sfrac{1}{2}} \norm{ \ubv}_{\underline{\cV}},
	\end{equation}
	For all $\btau\in \cH\cap H^s(\cT_h, \bbS)$ with $s\geq \sfrac{1}{2}$, and for all $\ubv\in \underline{\cV}$. The result follows by applying the discrete trace inequality \eqref{discTrace}.
	\end{proof}	
	
	Like in \eqref{var:golbal_h}, the semi-discrete problem \eqref{sd} also exhibits a differential-algebraic structure. This necessitates further analysis to guarantee its well-posedness.

\begin{proposition}
	Problem~\eqref{sd} admits a unique solution.
\end{proposition}

\begin{proof}
	The differential-algebraic system of equations \eqref{sd} can be written as follows
	\begin{align}\label{ADE1}
		\begin{split}
			&\inner{\rho \dot\bu_h, \bv}_{\cT_h} + \inner{\cA_s\dot{\bsig}_{h}, \btau }_{\Omega_s}  +  \inner*{\cA_s  \bsig_h, \btau}_{\Omega_f}  +
		  \inner{\bsig_h, \beps(\bv)}_{\cT_h} - \inner{\btau, \beps(\bu_h)}_{\cT_h} 
	\\
	&\quad - \dual{\bsig_h\bn, \bv}_{\partial \cT_h} +\dual{ \jump{\ubu_h}, \tfrac{(k+1)^2}{h_\cF}\bv + \btau\bn}_{\partial \cT_h} 
= \inner*{\bF(t), \bv}_{\cT_h}\quad \forall \btau \in \cH_h,\ \forall \bv \in \cP_{k+1}(\cT_h,\bbR^d)
\\
&\qquad \qquad \qquad \qquad \dual{\tfrac{(k+1)^2}{h_\cF} \hat\bu_h, \hat\bv  }_{\partial \cT_h} = \dual{\tfrac{(k+1)^2}{h_\cF} \bu_h - \bsig_h\bn, \hat\bv }_{\partial \cT_h}\quad \forall \hat\bv \in \cP_{k+1}(\cF^0_h, \bbR^d).
\end{split}
	\end{align}  
It follows from the second equation of \eqref{ADE1}, that  for all $t\in (0,T]$, $\hat\bu_h(t)$ can be expressed as follows:
\[
\hat\bu_h(t) = \mmean{\bu_h}(t) - \tfrac{h_\cF}{(k+1)^2}\mmean{\bsig_h\bn}(t)
\] 
where $\mmean{\bu_h}(t)$ and $\mmean{\bsig_h\bn}(t)$ are defined in $\cP_{k+1}(\cF^0_h,\bbR^d)$ by $\mmean{\bu_h}|_F := \frac12 (\bu_h|_K + \bu_h|_{K'})|_F$ and $\mmean{\bsig_h\bn}|_F := \frac12( \bsig_h|_K \bn_K + \bsig_h|_{K'}\bn_{K'})|_F$, respectively, for all $F\in \cF^0_h$ with $F= K \cap K'$. We can then eliminate $\hat\bu_h$ and end up with the system 
 \begin{align}\label{ADE2}
		\begin{split}
			\inner{\rho \dot\bu_h, \bv}_{\cT_h} + \inner{\cA_s\dot{\bsig}_{h}, \btau }_{\Omega_s}  +  \inner*{\cA_f \bsig_h, \btau}_{\Omega_f}  +
		  \inner{\bsig_h, \beps(\bv)}_{\cT_h} &- \inner{\btau, \beps(\bu_h)}_{\cT_h}- \dual{\bsig_h\bn, \bv}_{\partial \cT_h}
	\\
	 +\dual{ \bu_h - (\mmean{\bu_h} - \tfrac{h_\cF}{(k+1)^2}\mmean{\bsig_h\bn}), \tfrac{(k+1)^2}{h_\cF}\bv + \btau\bn}_{\partial \cT_h} 
&= \inner*{\bF(t), \bv}_{\cT_h},
\end{split}
\end{align}
for all $\btau \in \cH_h$ and $\bv \in \cP_{k+1}(\cT_h,\bbR^d)$. 

On the other hand, the orthogonal decomposition $\cH_h = \cH_{s,h} \oplus \cH_{f,h}$, defined in terms of the subspaces  
\[
\cH_{s,h} := \set{\btau \in \cH_h;\ \btau|_{\Omega_f} = \mathbf 0}\quad  \text{and} \quad \cH_{f,h} := \set{\btau \in \cH_h;\ \btau|_{\Omega_s} = \mathbf 0},
\]
induces the splitting $\bsig_h(t) = \bsig_{s,h} + \bsig_{f,h}$, with $\bsig_{s,h}(t)\in \cH_{s,h}$ and $\bsig_{f,h}(t)\in \cH_{f,h}$. Choosing $\bv = 0$ and $\btau \in  \cH_{f,h}$ in \eqref{ADE2} yields
\begin{equation}\label{ADE22}
		\inner{\cA_f\bsig_{f,h}, \btau}_{\Omega_f} - \inner{\btau, \beps(\bu_h)}_{\cT_h}
		+\dual{ \bu_h - (\mmean{\bu_h} - \tfrac{h_\cF}{(k+1)^2}\mmean{\bsig_{f,h}\bn}),  \btau\bn}_{\partial \cT_h} = 0\quad \forall \btau \in \cH_{f,h}.
\end{equation}
Given $\bv_h\in \cP_{k+1}(\cT_h,\bbR^d)$, we define $R_h \bv_h\in \cH_{f,h}$ as the unique solution of the linear system of equations
\[
	\inner{\cA_f(R_h \bv_h), \btau}_{\Omega_f} + 2 \dual{ \tfrac{h_\cF}{(k+1)^2}\mmean{(R_h \bv_h)\bn}),  \mmean{\btau\bn}}_{ \cF_h} = \dual{ \mmean{\bv_h} - \bv_h   ,  \btau\bn}_{\partial \cT_h} 
	+ \inner{\btau, \beps(\bv_h)}_{\cT_h}\quad \forall \btau \in \cH_{f,h}.	
\]
It follows from \eqref{ADE22} that the component $\bsig_{f,h}(t)\in \cH_{f,h}$ of  the solution $\bsig_h(t)$ is given by $\bsig_{f,h}(t) = R_h (\bu_h(t))$.
Using this fact in \eqref{ADE2} we conclude that the component $\bsig_{s,h}(t)\in \cH_{s,h}$ of $\bsig_h(t)$  and $\bu_h\in \cC^1(\cP_{k+1}(\cT_h,\bbR^d))$ are the unique solutions of the system of ordinary differential equations
\begin{align}\label{ADE2R}
	\begin{split}
		&\inner{\rho \dot\bu_h, \bv}_{\cT_h} + \inner{\cA_s\dot{\bsig}_{s,h}, \btau }_{\Omega_s} =    -
		\inner{\bsig_{s,h} + R_h\bu_h, \beps(\bv)}_{\cT_h}   + \inner{\btau, \beps(\bu_h)}_{\cT_h}
\\
&\quad + \dual{ (\bsig_{s,h} + R_h\bu_h)\bn, \bv}_{\partial \cT_h}  - \dual{ \tfrac{h_\cF}{(k+1)^2}\mmean{(\bsig_{s,h} + R_h\bu_h)\bn}, \tfrac{(k+1)^2}{h_\cF}\bv + \btau\bn}_{\partial \cT_h}
\\
&\quad \quad  + \dual{ \mmean{\bu_h} - \bu_h, \tfrac{(k+1)^2}{h_\cF}\bv + \btau\bn}_{\partial \cT_h} + \inner*{\bF(t), \bv}_{\cT_h}, \quad \forall (\bv, \btau)\in \cP_{k+1}(\cT_h,\bbR^d) \times \cH_{s,h},
\end{split}
\end{align}
with the initial conditions \eqref{initial-R1-R2-h*c}.

As $\inner{\cA_s\cdot, \cdot}_{\Omega_s} + \inner{\rho \cdot, \cdot}_{\cT_h}$ is an inner product in the finite dimensional space $\cH_{s,h}\times \cP_{k+1}(\cT_h,\bbR^d)$, we can select a suitable set of basis  for this function space to represent the semidiscrete problem \eqref{ADE2}  as a conventional first-order system of differential equations. As a result, the existence of a unique  solution $(\bsig_{s,h}, \bu_h)$ to \eqref{ADE2R} is guaranteed. It follows that the unique solution to \eqref{sd}-\eqref{initial-R1-R2-h*c} is given by $\bsig_h = \bsig_{s,h} + R_h(\bu_h)$ and $\ubu_h = (\bu_h, \hat \bu_h)$, with $\hat\bu_h = \mmean{\bu_h} - \tfrac{h_\cF}{(k+1)^2}\mmean{\bsig_h\bn}$. 
\end{proof}

The next step is to prove the consistency of the HDG scheme \eqref{sd} with respect to the continuous problem \eqref{var:golbal}.

\begin{proposition}\label{consistency}
Let
\[
   \bsig\in H^1_{[0, T]}(\cH)\cap L^2_{[0,T]}(H(\bdiv,\Omega,\bbS))
  \quad \text{and} \quad 
  \bu\in W^{1,\infty}_{[0,T]}(L^2(\Omega,\bbR^d))\cap L^2_{[0,T]}(H^1_0(\Omega,\bbR^d))
\]
be the solutions of \eqref{var:golbal}-\eqref{init0} and assume that the stress tensor $\bsig$ belongs  to  $\cC_{[0,T]}^0(H^s(\cT_h, \bbS)\cap H(\bdiv, \Omega, \bbS))$, with $s>\sfrac{1}{2}$. Then, it holds true that  
 \begin{align}\label{consistent}
\begin{split}
	\inner{\rho \dot\bu, \bv}_{\cT_h} &+ \inner{\cA_s\dot{\bsig}, \btau }_{\Omega_s}  +  \inner*{ \cA_f\bsig, \btau}_{\Omega_f}  
	\\
	&+ B_h(\bsig, \ubv) - B_h(\btau, \underline i(\bu)) +\dual{\tfrac{(k+1)^2}{h_\cF}\jump{\underline i(\bu)}, \jump{\ubv}}_{\partial \cT_h} 
= \inner*{\bF(t), \bv}_{\cT_h} 
\end{split} 
\end{align}
for all $\btau \in \mathcal H_h$ and $\ubv = (\bv, \hat \bv)\in \underline{\cV}_h$. 
\end{proposition}

\begin{proof}
The continuity of the normal components of $\bsig$ across the interelements of $\cT_h$ gives 
  \begin{equation*}
   B_h(\bsig,\ubv) = \inner{\bsig, \beps(\bv)}_{\cT_h} - \dual*{ \bsig\bn,  \jump{\ubv}}_{\partial \cT_h} = \inner{\bsig, \nabla \bv}_{\cT_h} - \dual*{ \bsig\bn,  \bv}_{\partial \cT_h}\quad \forall\ubv = (\bv, \hat \bv)\in \underline{\cV}_h.
  \end{equation*}
 Applying an elementwise integration by parts to the right-hand side  of the last identity followed by the substitution $\bdiv \bsig = \rho \dot\bu - \bF$ yields 
\begin{equation}\label{b2}
   B_h(\bsig,\ubv)  = - \inner{\bdiv\bsig, \bv}_{\cT_h} = -\inner{\rho\dot\bu, \bv}_{\cT_h} + \inner{\bF(t), \bv}_{\Omega}.
  \end{equation}
On the other hand, as $\jump{\underline i (\bu)} = \mathbf 0$, we have that 
\begin{equation}\label{b1}
  B_h(\btau, \underline i(\bu)) = \inner{ \btau, \beps(\bu) }_{\cT_h}\quad \forall \btau \in \cH_h.
\end{equation}
Substituting back \eqref{b2} and \eqref{b1} in \eqref{consistent}  gives  
\begin{align}\label{consistent0}
\nonumber
\inner{\rho \dot\bu, \bv}_{\cT_h} &+ \inner{\cA_s\dot{\bsig}, \btau }_{\Omega_s}  +  \inner*{ \cA_f\bsig, \btau}_{\Omega_f}  
+ B_h(\bsig, \ubv) - B_h(\btau, \underline i(\bu)) +\dual{\tfrac{(k+1)^2}{h_\cF}\jump{\underline i(\bu)}, \jump{\ubv}}_{\partial \cT_h} 
\\
&= \inner{\cA_s\dot{\bsig}, \btau }_{\Omega_s}  +  \inner*{ \cA_f\bsig, \btau}_{\Omega_f}  
	 - \inner{ \btau, \beps(\bu) }_{\cT_h} +  \inner{\bF, \bv}_{\cT_h},
\end{align}
for all $\btau\in \mathcal H_h$ and $\ubv = (\bv, \hat \bv)\in \underline{\cV}_h$. Moreover, applying Green's formula \eqref{green} in the second equation of \eqref{var:golbal} and keeping in mind the density of the embedding $H(\bdiv,\Omega,\bbS) \hookrightarrow \cH$ we obtain
\begin{equation*}
	\inner{\cA_s\dot{\bsig}, \btau }_{\Omega_s}  +  \inner*{ \cA_f\bsig, \btau}_{\Omega_f} = \inner*{\beps(\bu) , \btau }_\Omega 
\quad \forall \btau\in \cH.
 \end{equation*}
Using  this identity in \eqref{consistent0} gives the result.   
\end{proof}

\section{Convergence analysis of the HDG method}\label{sec:convergence}

The convergence analysis of the HDG method \eqref{sd} follows standard procedures. Using the stability of the HDG method and the consistency result \eqref{consistent}, we can estimate the projected errors $\bpi_h(t) :=\Pi_\cT^k\bsig(t) - \bsig_h(t)$ and $\ube_h(t):= \Xi^{k+1}_\cT \underline i (\bu(t)) - \ubu_h(t)$ in the stress and velocity variables, respectively, in terms of $\bchi(t) := \bsig(t) - \Pi\cT^k\bsig(t)$ and $\ubq(t) := \underline i (\bu(t)) - \Xi^{k+1}_\cT \underline i (\bu(t))$.
We notice that the components of $\ube_h:= (\be_h, \hat\be_h)\in \cC^1(\underline{\cV}_h)$ are given by
\begin{equation*}   
	 \be_h = \Pi^{k+1}_\cT \bu - \bu_h\quad \text{and}\quad   \hat\be_h = \Pi_\cF^{k+1}\hat\bu - \hat\bu_h,
\end{equation*}
where we recall that $\hat \bu = \bu|_{\cF_h}$. Similarly, $\ubq = (\bq, \hat\bq)$ with  
\begin{equation*}  
	 \bq = \bu - \Pi^{k+1}_\cT \bu  \quad \text{and}\quad   \hat\bq = \hat\bu - \Pi_\cF^{k+1}\hat\bu.
\end{equation*}

\begin{lemma}\label{stab_sd}
	Under the conditions of Proposition~\ref{consistency},  there exists a constant $C>0$ independent of $h$ and $k$ such that 
	\begin{equation}\label{stab}
	\int_{0}^{T}\norm{\bpi_h(t)}^2_\cH\, \mathrm{d}t  +  \max_{[0, T]}\norm{ \be_h(t)}^2_{0,\cT_h} + \int_0^T \norm{\tfrac{k+1}{h_\cF^{\sfrac{1}{2}}} \jump{\ube_h}  }^2_{0,\partial \cT_h}\mathrm{d}t
	 \leq C \int_0^T \Big(\norm{ \ubq}^2_{\underline{\cV}} + \norm{\tfrac{h_\cF^{\sfrac{1}{2}}}{k+1} \bchi\bn}^2_{0,\partial \cT_h}     \Big) \mathrm{d}t.
	\end{equation} 
	\end{lemma}
	\begin{proof}
	By virtue of the consistency property \eqref{consistent},  it is straightforward that 
	 \begin{align}\label{orthog} 
	\begin{split}
			  &\inner{\rho \dot\be_h, \bv}_{\cT_h} + \inner{\cA_s\dot{\bpi}_h, \btau }_{\Omega_s}  +   \inner{ \cA_f  \bpi_h, \btau}_{\Omega_f}  
		+ B_h(\bpi_h, \ubv) - B_h(\btau, \ube_h) +\dual{\tfrac{(k+1)^2}{h_\cF}\jump{\ube_h}, \jump{\ubv}}_{\partial \cT_h}
		\\ &
	\qquad = -  \inner{\rho \dot\bq, \bv}_{\cT_h} -\inner{\cA_s\dot{\bchi}, \btau }_{\Omega_s}   -\inner*{ \cA_f  \bchi, \btau}_{\Omega_f}  
		- B_h(\bchi, \ubv) + B_h(\btau, \ubq) -\dual{\tfrac{(k+1)^2}{h_\cF}\jump{\bchi}, \jump{\ubv}}_{\partial \cT_h}
		\\ &
	\qquad  = \dual{\bchi\bn, \jump{\ubv}}_{\partial \cT_h} + B_h(\btau, \ubq) - \dual{\tfrac{(k+1)^2}{h_\cF}\jump{\ubq}, \jump{\ubv}}_{\partial \cT_h}, \quad \forall \btau\in \mathcal H_h\ \forall \ubv = (\bv, \hat \bv)\in \underline{\cV}_h.
	\end{split} 
	\end{align}
	The last identity follows from the orthogonal properties  
	\begin{equation*}
		  \inner{\cA_s\dot{\bchi}, \btau }_{\Omega_s} + \inner{ \cA_f  \bchi, \btau}_{\Omega_f}  = \inner{\dot\bsig - \Pi_\cT^k \dot\bsig, \cA_s\btau_E}
	  +  \inner{ \cA_f \bsig - \Pi_\cT^k \bsig, \cA_f\btau}_{\Omega_f}
	  = 0 \quad \forall \btau \in \cH_h,
	\end{equation*}
	 \[
	  \inner{\rho \dot\bq, \bv}_{\cT_h} = \inner{\dot\bu - \Pi_\cT^{k+1}\dot\bu, \rho\bv}_{\Omega} =0 \quad \forall \bv\in \cP_{k+1}(\cT_h, \bbR^d), 
	\]
	and  
	\[
	  B_h(\bchi, \ubv) = \inner{ \bsig - \Pi_\cT^k \bsig, \beps(\bv) }_{\cT_h} - \dual{\bchi\bn, \jump{\ubv}}_{\partial \cT_h} = - \dual{\bchi\bn, \jump{\ubv}}_{\partial \cT_h}\quad \forall \ubv \in \underline{\cV}_h,
	\]
	where we took into account here the inclusion $\beps(\cP_{k+1}(\cT_h, \bbR^d)) \subset \cP_k(\cT_h, \bbS)$.

	 The choices $\btau = \bpi_h$ and $\ubv = \ube_h$ in \eqref{orthog} and the Cauchy-Schwarz inequality together with \eqref{Bhh} yield 
	\begin{align*}
		\frac12 \frac{\text{d}}{\text{d}t} &\big\{\norm{\rho^{\sfrac12} \be_h}^2_{\cT_h} + \norm{\cA_s^{\sfrac12}\bpi_h}^2_{0,\Omega_s}  \big\} + \norm{ \cA_f^{\sfrac12}  \bpi_h}^2_{0,\Omega_f} + \norm{\tfrac{k+1}{h_\cF^{\sfrac{1}{2}}}\jump{\ube_h}}^2_{0,\partial \cT_h}
		\\
		&\leq 
		\dual{\bchi\bn, \jump{\ube_h}}_{\partial \cT_h} 
		+ B_h(\bpi_h, \ubq) - \dual{\tfrac{(k+1)^2}{h_\cF}\jump{\ubq}, \jump{\ube_h}}_{\partial \cT_h}
		\\
		&\leq \norm{\tfrac{h_\cF^{\sfrac{1}{2}}}{k+1} \bchi\bn}_{0,\partial \cT_h} 
		\norm{\tfrac{k+1}{h_\cF^{\sfrac{1}{2}}} \jump{\ube_h}}_{0,\partial \cT_h} + C  \norm{\bpi_h}_\cH  \norm{ \ubq}_{\underline{\cV}} 
		 + \norm{\tfrac{k+1}{h_\cF^{\sfrac{1}{2}}} \jump{\ubq}  }_{0,\partial \cT_h} 
		\norm{\tfrac{k+1}{h_\cF^{\sfrac{1}{2}}} \jump{\ube_h}  }_{0,\partial \cT_h}. 
	\end{align*}
	We notice that, because of assumption \eqref{initial-R1-R2-h*c}, the projected errors  satisfy vanishing initial conditions, namely, $\bpi_h(0) = \mathbf 0$  and $\ube_{h}(0) = (\mathbf 0, \mathbf 0)$. Hence, integrating over $t\in (0, T]$ and using again the Cauchy-Schwarz inequality we deduce that 
	\begin{align*}
		&\norm{\rho^{\sfrac{1}{2}} \be_h(t)}^2_{0,\cT_h}  + \norm{\cA_s^{\sfrac12}\bpi_h(t)}^2_{0,\Omega_s} + \int_0^t \norm{\cA_f^{\sfrac12} \bpi_h(s) }^2_{0,\Omega_f} \,\text{d}s+ \int_0^t \norm{\tfrac{k+1}{h_\cF^{\sfrac{1}{2}}} \jump{\ube_h}  }^2_{0,\partial \cT_h}\,\text{d}s
		\\
		&\qquad \lesssim \Big(\int_0^T (\norm{\tfrac{h_\cF^{\sfrac{1}{2}}}{k+1} \bchi\bn}^2_{0,\partial \cT_h}  +  \norm{\tfrac{k+1}{h_\cF^{\sfrac{1}{2}}} \jump{\ubq}  }^2_{0,\partial \cT_h})  \text{d}t\Big)^{\sfrac{1}{2}}  
		\Big( \int_0^T \norm{\tfrac{k+1}{h_\cF^{\sfrac{1}{2}}} \jump{\ube_h}  }^2_{0,\partial \cT_h}\ \text{d}t\Big)^{\sfrac{1}{2}}
		\\
		&\qquad \qquad + \Big(\int_{0}^{T}\norm{\bpi_h(t)}^2_\cH \text{d}t\Big)^{\sfrac12} \Big( \int_0^T \norm{ \ubq}^2_{\underline{\cV}} \text{d}t \Big)^{\sfrac12},\quad \forall t\in (0, T].
		 \end{align*}
	Finally, a simple application of Young's inequality yields 
	\begin{align*}
		 &\max_{[0, T]}\norm{ \be_h(t)}^2_{0,\cT_h} + \max_{[0,T]}\norm{\cA_s^{\sfrac12}\bpi_h(t)}^2_{0,\Omega_s}  +  \int_0^T \norm{\cA_f^{\sfrac12} \bpi_h(t) }^2_{0,\Omega_f} \text{d}t
		\\
		&\qquad \qquad + \int_0^T \norm{\tfrac{k+1}{h_\cF^{\sfrac{1}{2}}} \jump{\ube_h}  }^2_{0,\partial \cT_h}\,\text{d}t \lesssim \int_0^T \Big(\norm{\tfrac{h_\cF^{\sfrac{1}{2}}}{k+1} \bchi\bn}^2_{0,\partial \cT_h}  +  \norm{ \ubq}^2_{\underline{\cV}}\Big) \, \text{d}t,
	\end{align*}
	and the result follows.
	\end{proof}
	
	As a consequence of the stability estimate \eqref{stab}, we immediately have the following convergence result for the HDG method \eqref{var:golbal_h}-\eqref{init0_h}.
	\begin{theorem}\label{hpConv}
		Let 
		\[
   \bsig\in H^1_{[0, T]}(\cH)\cap L^2_{[0,T]}(H(\bdiv,\Omega,\bbS))
  \quad \text{and} \quad 
  \bu\in W^{1,\infty}_{[0,T]}(L^2(\Omega,\bbR^d))\cap L^2_{[0,T]}(H^1_0(\Omega,\bbR^d))
  \]
		be the solutions of \eqref{var:golbal}-\eqref{init0}. Assume that $\bsig \in L^2_{[0,T]}(H^{1+r}(\Omega_s\cup \Omega_f, \bbS))$ and $\bu \in \cC^0_{[0,T]}(H^{2+r}( \Omega_s\cup \Omega_f, \bbR^d))$, with $r\geq 0$. Then, there exists a constant $C>0$ independent of $h$ and $k$ such that 
		\begin{align*}
			 &\max_{[0, T]}\norm{(\bu - \bu_h)(t)}_{0,\cT_h}  + \left( \int_{0}^{T}\norm{(\bsig - \bsig_h)(t)}^2_{\cH}\mathrm{d}t\right)^{\sfrac{1}{2}}  + \left( \int_0^T \norm{\tfrac{k+1}{h_\cF^{\sfrac{1}{2}}}\jump{\underline i(\bu) - \ubu_h}}^2_{0,\partial \cT_h}\,\mathrm{d}t\right)^{\sfrac{1}{2}}
			\\
			&\quad \leq  C \tfrac{h_K^{\min\{ r, k \}+1}}{(k+1)^{r+\sfrac12}} \Big( \max_{[0,T]}\norm{\bu}^2_{2+r,\Omega_j} + \int_{0}^{T}\norm*{\btau}^2_{1+r,\Omega_s\cup \Omega_f} \mathrm{d}t  \Big)^{\sfrac{1}{2}}\quad \forall k\geq 0. 
			\end{align*}
		\end{theorem}
		\begin{proof}
			It follows from the triangle inequality and \eqref{stab} that 
		\begin{align*}
			 &\max_{[0, T]}\norm{(\bu - \bu_h)(t)}^2_{0,\cT_h}  +  \int_{0}^{T}\norm{(\bsig - \bsig_h)(t)}^2_{\cH} \text{d}t + \int_0^T \norm{\tfrac{k+1}{h_\cF^{\sfrac{1}{2}}}\jump{\underline i(\bu) - \ubu_h}}^2_{0,\partial \cT_h}\,\text{d}t
			\\
			&\quad \leq \max_{[0, T]}\norm{\bq(t)}^2_{0,\cT_h} + \int_{0}^{T}\norm{\bchi(t)}^2_{\cH}\text{d}t  +     C \int_0^T \Big(\norm{ \ubq}^2_{\underline{\cV}} + \norm{\tfrac{h_\cF^{\sfrac{1}{2}}}{k+1} \bchi\bn}^2_{0,\partial \cT_h}     \Big) \ \mathrm{d}t, 
		\end{align*}
		and the result follows directly from the error estimates \eqref{tool1} and \eqref{tool2}.
		\end{proof}
		
\begin{remark}\label{R1}
 The energy norm error estimates provided by Theorem~\ref{hpConv} are quasi-optimal in $h$ and suboptimal by a factor of $k^{\sfrac{1}{2}}$ in $k$. This issue has been previously addressed in the literature, as shown in \cite{houston2002}, where similar bounds were obtained for stationary second-order elliptic problems. 
 \end{remark}

\section{The fully discrete scheme and its convergence analysis}\label{sec:fully-discrete}

Given $L\in \mathbb{N}$, we consider a uniform partition of the time interval $[0, T]$ with step size $\Delta t := T/L$ and nodes $t_n := n\,\Delta t$, $n=0,\ldots, L$. The midpoint of each time subinterval is represented as $t_{n+\sfrac{1}{2}}:= \frac{t_{n+1} + t_n}{2}$. For any finite sequence  $\set{\phi^n,\ n=0,\ldots,L}$ of real numbers,  we define  $\mean{\phi^n}:= \frac{\phi^{n+1} + \phi^n}{2}$ and  introduce the discrete time derivative $\partial_t \phi^n := \frac{\phi^{n+1} - \phi^n}{\Delta t}$. We adopt the same notation for sets of vectors or tensors.

In what follows we utilize the Crank-Nicolson method for the time discretisation of 
\eqref{sd}-\eqref{initial-R1-R2-h*c}. Namely, for  $n=0,\ldots,L-1$, 
we seek  $\bsig^{n+1}_h \in  \mathcal{H}_h$ and  $\ubu^{n+1}_h := (\bu^{n+1}_h, \hat\bu^{n+1}_h) \in \underline{\cV}_h$  solution of 
\begin{align}\label{fd}
\begin{split}
	\inner{\rho \partial_t\bu^n_h, \bv}_{\cT_h} &+ \inner{\cA_s\partial_t \bsig^n_h, \btau }_{\Omega_s}  +  \inner{ \cA_f  \mean{\bsig^{n}_h},  \btau}_{\Omega_f}  
	\\
	&+ B_h(\mean{\bsig^{n}_h}, \ubv) - B_h(\btau, \mean{\ubu^{n}_h}) +\dual{ \tfrac{(k+1)^2}{h_\cF}\jump{\mean{\ubu^{n}_h}}, \jump{\ubv}}_{\partial \cT_h} 
= \inner*{ \mean{\bF(t_{n})}, \bv}_{\cT_h}
\end{split}
\end{align}
for all $\btau\in \mathcal H_h$ and $\ubv = (\bv, \hat \bv)\in \underline{\cV}_h$. We assume that the scheme \eqref{fd} is initiated  with 
 \begin{equation}\label{initial-fd}
	\ubu_h^0 = \Xi^{k+1}_\cT \underline i (\bu^0) \quad \text{and} \quad 	\bsig_h^0|_{\Omega_s}= \Pi_\cT^k \bsig^0_s.
 \end{equation}  
 
We point out that each iteration step of \eqref{fd} requires solving a square system of linear equations whose matrix stems from the bilinear form 
\begin{align*}
	 \tfrac{1}{\Delta t} \inner{\rho\bu_h, \bv_h}_{\cT_h} &+ \tfrac{1}{\Delta t} \inner{\cA_s \bsig_h, \btau_h}_{\Omega_s} +  \tfrac{1}{2}  \inner{ \cA_f  \bsig_h, \omega \btau}_{\Omega_f} 
	\\
	&+ \tfrac{1}{2} B_h(\bsig_h, \ubv_h) - \tfrac{1}{2} B_h(\btau_h, \ubu_h) + \tfrac{1}{2}\dual{ \tfrac{(k+1)^2}{h_\cF}\jump{\ubu_h}, \jump{\ubv}}_{\partial \cT_h}.
\end{align*} 
The coerciveness of this bilinear form on $(\underline{\cV}_h \times \cH_h)\times (\underline{\cV}_h \times \cH_h)$ ensures the well-defined nature of the scheme \eqref{fd}-\eqref{initial-fd}. 

Our aim now is to obtain a fully discrete counterpart of Lemma~\ref{stab_sd}. We recall that, according to our notations, the components of the projected errors $\bpi_h^n\in \cH_h$ and $\ube_h^n:= (\be_h^n, \hat\be_h^n)\in \underline{\cV}_h$ are expressed as 
\begin{align*}
	\bpi_h^n =  \Pi_\cT^k\bsig(t_n) - \bsig_h^n,\quad   
	 \be_h^n = \Pi^{k+1}_\cT \bu(t_n) - \bu_h^n,\quad   \hat\be_h^n = \Pi_\cF^{k+1}\hat\bu(t_n) - \hat\bu_h^n,
\end{align*}
where $\hat \bu = \bu|_{\cF_h}$, while $\bchi^n$ and $\ubq^n = (\bq^n, \hat\bq^n)$ are given by  
\begin{align*}
	\bchi^n =  \bsig(t_n) - \Pi_\cT^k\bsig(t_n),\quad 
	 \bq^n = \bu(t_n) - \Pi^{k+1}_\cT \bu(t_n),\quad   \hat\bq^n = \hat\bu(t_n) - \Pi_\cF^{k+1}\hat\bu(t_n).
\end{align*}

\begin{lemma}\label{coco}
	Let 
\[
   \bsig\in H^1_{[0, T]}(\cH)\cap L^2_{[0,T]}(H(\bdiv,\Omega,\bbS))
  \quad \text{and} \quad 
  \bu\in W^{1,\infty}_{[0,T]}(L^2(\Omega,\bbR^d))\cap L^2_{[0,T]}(H^1_0(\Omega,\bbR^d))
\]
be the solutions of \eqref{var:golbal}-\eqref{init0} and assume that the stress tensor $\bsig$ belongs  to  $\cC_{[0,T]}^0(H^r(\cT_h, \bbS)\cap H(\bdiv, \Omega, \bbS))$, with $r>\sfrac{1}{2}$.   There exists a constant $C>0$ independent of $h$, $k$ and $\Delta t$ such that 
	\begin{align}\label{stabE}
		\begin{split}
			  &\max_n\norm{ \rho^{\sfrac{1}{2}}\be^{n}_h }_{0,\cT_h}^2  + \max_n \norm{ \cA_s^{\sfrac12}\bpi^{n}_h }_{0,\Omega_s}^2 + \Delta t  \sum_{n=0}^{L-1} \norm{\cA_f^{\sfrac12} \mean{\bpi^{n}_h}}^2_{0,\Omega_f} + \Delta t  \sum_{n=0}^{L-1}\norm{ \tfrac{(k+1)}{h_\cF^{\sfrac{1}{2}}}\jump{\mean{\ube^{n}_h}}}^2_{0,\partial \cT_h} 
			 \\ 
			  &\qquad \leq C \Big(  \Delta t\sum_{n=0}^{L-1} \norm{ \partial_t\bsig(t_n) - \mean{\dot\bsig(t_n)}}_\cH^2 + \Delta t\sum_{n=0}^{L-1} \norm{\rho^{\sfrac{1}{2}}(\partial_t \bu(t_n) -  \mean{ \dot{\bu}(t_{n})})}_{0,\cT_h}^2
			\\
			&  
			 \qquad \qquad \qquad \qquad  + \Delta t\sum_{n=0}^{L-1} \norm{\tfrac{h_\cF^{\sfrac{1}{2}}}{k+1} \mean{\bchi^{n}}\bn}_{0,\partial \cT_h}^2
		+ \Delta t\sum_{n=0}^{L-1}  \norm{\mean{\ubq^n}}_{\underline{\cV}}^2 \Big).
		\end{split}
	\end{align}
	\end{lemma}
	\begin{proof}
	It follows from the consistency equation \eqref{consistent} and the orthogonality properties employed in the proof of Lemma~\ref{stab_sd}  that the projected errors $\bpi_h^n\in \cH_h$ and $\ube_h^n\in \underline{\cV}_h$ satisfy the equation  
	 \begin{align}\label{proE}
	\begin{split}
		&\inner{\rho \partial_t\be_h^n, \bv}_{\cT_h} + \inner{\cA_s\partial_t \bpi_h^n, \btau }_{\Omega_s}  +  \inner{\cA_f \mean{\bpi^{n}_h},  \btau}_{\Omega_f}  
		+ B_h(\mean{\bpi^{n}_h}, \ubv) - B_h(\btau, \mean{\ube^{n}_h}) 
		\\
	&\qquad +\dual{ \tfrac{(k+1)^2}{h_\cF}\jump{\mean{\ube^{n}_h}}, \jump{\ubv}}_{\partial \cT_h} = \inner{\rho (\partial_t \bu(t_n) -  \mean{ \dot{\bu}(t_{n})}), \bv}_{\cT_h} + \inner{\cA_s (\partial_t\bsig(t_n) - \mean{\dot\bsig(t_n)}), \btau}_{\Omega_s} 
		\\
		&\qquad \qquad + \dual{\mean{\bchi^{n}}\bn, \jump{\ubv}}_{\partial \cT_h}+ B_h(\btau, \mean{\ubq^{n}}) -\dual{ \tfrac{(k+1)^2}{h_\cF}\jump{\mean{\ubq^{n}}}, \jump{\ubv}}_{\partial \cT_h}
	\end{split}
	\end{align}
	for all $\btau\in \mathcal H_h$ and $\ubv = (\bv, \hat \bv)\in \underline{\cV}_h$. 
	Selecting $\btau = \mean{\bpi^n_h}$ and $\ubv = \mean{\ube^n_h}$ in equation \eqref{proE} and applying \eqref{Bhh} and the Cauchy-Schwartz inequality to the terms on the right-hand side we derive the estimate 
	\begin{align*}
		\begin{split}
			&\frac{1}{2\Delta t}\left(\norm{ \rho^{\sfrac{1}{2}}\be^{n+1}_h }_{0,\cT_h}^2 +  \norm{ \cA_s^{\sfrac12}\bpi^{n+1}_h }_{0,\Omega_s}^2 -  \norm{ \rho^{\sfrac{1}{2}}\be^{n}_h }_{0,\cT_h}^2   - \norm{ \cA_s^{\sfrac12}\bpi^{n}_h }_{0,\Omega_s}^2  \right) 
			\\
			 &\quad + \norm{\cA_f^{\sfrac12} \mean{\bpi^{n}_h}}^2_{0,\Omega_f}  
			+ \norm{ \tfrac{k+1}{h_\cF^{\sfrac{1}{2}}}\jump{\mean{\ube^{n}_h}}}^2_{0,\partial \cT_h}  
			\\
			&\quad \leq \norm{\rho^{\sfrac{1}{2}}(\partial_t \bu(t_n) -  \mean{ \dot{\bu}(t_{n})})}_{0,\cT_h} \norm{ \rho^{\sfrac{1}{2}}\mean{\be^{n}_h} }_{0,\cT_h} + \norm{\cA_s^{\sfrac12} (\partial_t\bsig(t_n) - \mean{\dot\bsig(t_n)} ) }_{0,\Omega_s} \norm{\cA_s^{\sfrac12}\mean{\bpi^n_h}}_{0,\Omega_s}
			\\
			&\quad     + \norm{\tfrac{h_\cF^{\sfrac{1}{2}}}{k+1} \mean{\bchi^{n}}\bn}_{0,\partial \cT_h} \norm{ \tfrac{k+1}{h_\cF^{\sfrac{1}{2}}}\jump{\mean{\ube^{n}_h}}}_{0,\partial \cT_h}
		  + \norm{\mean{\bpi^n_h}}_\cH \norm{\mean{\ubq^n}}_{\underline{\cV}} + \norm{ \tfrac{k+1}{h^{\sfrac{1}{2}}_\cF}\jump{\mean{\ubq^{n}}}}_{0,\partial \cT_h} \norm{ \tfrac{k+1}{h^{\sfrac{1}{2}}_\cF}\jump{\mean{\ube^n_h}}}_{0,\partial \cT_h}.
		\end{split}
	\end{align*}
	Summing in the index $n$ and taking into account that the projected errors vanish  identically at the  initial step we  get 
	\begin{align*}
		\begin{split}
			\max_n\norm{ \rho^{\sfrac{1}{2}}\be^{n}_h }_{0,\cT_h}^2  &+ \max_n \norm{ \cA_s^{\sfrac12}\bpi^{n}_h }_{0,\Omega_s}^2 + \Delta t  \sum_{n=0}^{L-1} \norm{\cA_f^{\sfrac12} \mean{\bpi^{n}_h}}^2_{0,\Omega_f} + \Delta t  \sum_{n=0}^{L-1}\norm{ \tfrac{(k+1)}{h_\cF^{\sfrac{1}{2}}}\jump{\mean{\ube^{n}_h}}}^2_{0,\partial \cT_h} 
			\\
			& \lesssim \max_n \norm{ \cA_s^{\sfrac12}\bpi^{n}_h }_{0,\Omega_s} \Big(\Delta t\sum_{n=0}^{L-1} \norm{  \cA_s^{\sfrac12} ( \partial_t\bsig(t_n) - \mean{\dot\bsig(t_n)} ) }_{0,\Omega_s} \Big) 
			\\
			 &+ \max_n\norm{ \rho^{\sfrac{1}{2}}\be^{n}_h }_{0,\cT_h} \Big( \Delta t\sum_{n=0}^{L-1} \norm{\rho^{\sfrac{1}{2}}(\partial_t \bu(t_n) -  \mean{ \dot{\bu}(t_{n})})}_{0,\cT_h}   \Big)
		\\
		& + \Delta t\sum_{n=0}^{L-1} \norm{\tfrac{h_\cF^{\sfrac{1}{2}}}{k+1} \mean{\bchi^{n}}\bn}_{0,\partial \cT_h} \norm{ \tfrac{k+1}{h_\cF^{\sfrac{1}{2}}}\jump{\mean{\ube^{n}_h}}}_{0,\partial \cT_h}  
		+   \Delta t\sum_{n=0}^{L-1} \norm{ \bpi^{n}_h }_\cH \norm{\mean{\ubq^n}}_{\underline{\cV}} 
		\\
		&+ \Delta t\sum_{n=0}^{L-1} \norm{ \tfrac{k+1}{h^{\sfrac{1}{2}}_\cF}\jump{\mean{\ubq^{n}}}}_{0,\partial \cT_h} \norm{ \tfrac{k+1}{h^{\sfrac{1}{2}}_\cF}\jump{\mean{\ube^n_h}}}_{0,\partial \cT_h}.
		\end{split}
	\end{align*}
	We notice that 
	\[
		\Delta t\sum_{n=0}^{L-1} \norm{ \bpi^{n}_h }_\cH \norm{\mean{\ubq^n}}_{\underline{\cV}}
		\leq 
		\Big(\max_n \norm{ \cA_s^{\sfrac12}\bpi^{n}_h }_{0,\Omega_s}^2 + \Delta t  \sum_{n=0}^{L-1} \norm{\cA_f^{\sfrac12} \mean{\bpi^{n}_h}}^2_{0,\Omega_f} \Big)^{\sfrac12} \Big(\Delta t\sum_{n=0}^{L-1} 
		\norm{\mean{\ubq^n}}^2_{\underline{\cV}} \Big)^{\sfrac12}.
	\]
	Plugging this estimate back into the previous inequality and applying the Cauchy-Schwartz inequality along with  Young's inequality $2a b \leq \frac{a^2}{\epsilon} +  \epsilon b^2$, where a suitable $\epsilon > 0$ is chosen in each instance, yields \eqref{stabE}. 
	\end{proof}
	
	To handle the time-consistency terms in \eqref{stabE}, we rely on a Taylor expansion centered at $t=t_{n+\sfrac{1}{2}}$, yielding the identity:
	\[
	  \partial_t\varphi(t_n) = \mean{\dot\varphi(t_n)} + \frac{(\Delta t)^2}{16} \int_{-1}^1 \dddot{\varphi} (t_{n+\sfrac{1}{2}} + \tfrac{\Delta t}{2} s) (|s|^2 -1) \,  \text{d}s\quad \forall \varphi \in \cC^3([0, T]), \quad 0\leq n\leq L-1.
	\]
	Therefore, under the additional assumptions $\bsig \in \cC^{3}_{[0,T]}(\cH)$ and $\bu \in \cC^{3}_{[0,T]}(L^2(\Omega,\bbR^d))$, we deduce from \eqref{stabE} that 
	\begin{align}\label{stabET}
		\begin{split}
			  &\max_n\norm{ \rho^{\sfrac{1}{2}}\be^{n}_h }_{0,\cT_h}^2  + \max_n \norm{ \cA_s^{\sfrac12}\bpi^{n}_h }_{0,\Omega_s}^2 + \Delta t  \sum_{n=0}^{L-1} \norm{\cA_f^{\sfrac12} \mean{\bpi^{n}_h}}^2_{0,\Omega_f} + \Delta t  \sum_{n=0}^{L-1}\norm{ \tfrac{(k+1)}{h_\cF^{\sfrac{1}{2}}}\jump{\mean{\ube^{n}_h}}}^2_{0,\partial \cT_h}
			  \\
			&  \qquad\lesssim (\Delta t)^4 \big( \max_{[0,T]}\norm{\dddot{\bsig}}^2_{0,\Omega_s} + \max_{[0,T]}\norm{\dddot{\bu}}^2_{0,\cT_h}\big)
				+ \max_n\norm{\tfrac{h_\cF^{\sfrac{1}{2}}}{k+1} \bchi^{n}\bn}_{0,\partial \cT_h}^2
		+ \max_n  \norm{\ubq^n}^2_{\underline{\cV}}.
		\end{split}
	\end{align}
	
	We can now state the convergence result for the fully discrete scheme \eqref{fd}-\eqref{initial-fd}.
	\begin{theorem}\label{hpConvFD}
	Assume that the solution $(\bsig, \ubu)$ of \eqref{var:golbal}-\eqref{init0} satisfies the time regularity assumptions $\bsig\in H^1_{[0, T]}(H(\bdiv,\Omega,\bbS))\cap\cC^3_{[0,T]}(\cH)$,  $\bu \in \cC^3_{[0,T]}(L^2(\Omega,\bbR^d))\cap \cC^0_{[0,T]}(H^1_0(\Omega,\bbR^d))$ and the piecewise space regularity conditions $\bsig \in \cC^0(H^{1+r}(\Omega_s\cup \Omega_f, \bbS))$ and $\bu \in \cC^0(H^{2+r}(\Omega_s\cup \Omega_f, \bbR^d))$, with $r\geq 0$. Then, there exists a constant $C>0$ independent of $h$ and $k$ such that, for all $k\geq 0$,  
	\begin{align*}
		 &\max_n\norm{\bu(t_n) - \bu_h^n}_{0,\cT_h}  +  \Big(\Delta t  \sum_{n=0}^{L-1} \norm{ \bsig(t_n) - \mean{\bsig_h^n}}^2_{\cH}\Big)^{\sfrac{1}{2}}
		  + \Big(\Delta t \sum_{n=0}^{L-1} \norm{\tfrac{k+1}{h_\cF^{\sfrac{1}{2}}}\jump{\underline i(\bu(t_n)) - \mean{\ubu_h^n}}}^2_{0,\partial \cT_h}\,\text{d}t\Big)^{\sfrac{1}{2}}
		\\
		&\, \leq  C \tfrac{h_K^{\min\{ r, k \}+1}}{(k+1)^{r+\sfrac{1}{2}}} \Big(\max_{[0,T]}\norm*{\btau}_{1+r,\Omega_s\cup \Omega_f}  + \max_{[0,T]}\norm{\bu}_{2+r,\Omega_s\cup \Omega_f}\Big)
		 + C(\Delta t)^2 \big(\max_{[0,T]}\norm{\dddot{\bsig}}_{\cH} + \max_{[0,T]}\norm{\dddot{\bu}}_{0,\cT_h}\big). 
	\end{align*}
	\end{theorem}
	\begin{proof}
		Applying the triangle inequality  we deduce from \eqref{stabET} that 
	\begin{align*}
		&\max_n\norm{\bu(t_n) - \bu_h^n}_{0,\cT_h}  +  \Big(\Delta t  \sum_{n=0}^{L-1} \norm{ \bsig(t_n) - \mean{\bsig_h^n} }^2_{\cH}\Big)^{\sfrac{1}{2}}
		  + \Big( \Delta t \sum_{n=0}^{L-1} \norm{\tfrac{k+1}{h_\cF^{\sfrac{1}{2}}}\jump{\underline i(\bu(t_n)) - \mean{\ubu_h^n}}}^2_{0,\partial \cT_h}\,\text{d}t\Big)^{\sfrac{1}{2}}
		\\
		&\quad \lesssim \max_n  \norm{\bchi^n}_\cH + \max_n\norm{\tfrac{h_\cF^{\sfrac{1}{2}}}{k+1} \bchi^{n}\bn}_{0,\partial \cT_h}^2
		+ \max_n  \norm{\ubq^n}_{\underline{\cV}}
		 + C(\Delta t)^2 \big( \max_{[0,T]}\norm{\dddot{\bsig}}_{\cH} + \max_{[0,T]}\norm{\dddot{\bu}}_{0,\cT_h}\big)
		\end{align*}
		and the result is a direct consequence of the error estimates \eqref{tool1} and \eqref{tool2}.
	\end{proof}

	\section{Numerical results}\label{sec:numresults}

	The numerical results presented in this section have been implemented using the finite element library \texttt{Netgen/NGSolve} \cite{schoberl2014c++}. We validate our HDG scheme by reproducing select numerical results outlined in \cite{Fu2022}, where a different monolithic HDG approach is used to solve the same FSI linear model.

\subsection{Example 1: Validation of the convergence rates} 

In this example, we confirm the decay of error as predicted by Theorem~\ref{hpConvFD} with respect to the parameters $h$, $\Delta t$ and $k$.  We compare the computed solutions at different levels of refinement to an exact solution of problem \eqref{motion}-\eqref{initf} given by 
\begin{align}\label{exactSol}
\begin{split}
	 \bu_f(x,y,t) &:=
	\begin{pmatrix}
		\sin^2(2 \pi x) \sin( \frac{8}{3} \pi (y + 1)) \sin(2 t)
	\\
	-\frac{3}{2} \sin(4 \pi x)  \sin^2( \frac{4}{3} \pi (y + 1) ) \sin(2 t)
	\end{pmatrix} \quad \text{in $\Omega_f\times (0, T]$},
	\\
p(x,y,t) &:= \sin(2\pi x) \sin(2 \pi y) \sin(t) \quad \text{in $\Omega_f\times (0, T]$},
\\
\bd(x,y,t) &:= \begin{pmatrix}
	\sin^2(2\pi x) \sin( \frac{8}{3} \pi (y + 1)) \sin^2(t)
	\\
	-\frac{3}{2} \sin(4\pi x) \sin^2( \frac{4}{3} pi (y + 1) ) \sin^2 (t)
	\end{pmatrix} \quad \text{in $\Omega_s\times (0, T]$},
\end{split}
\end{align}
where $\Omega_f = (0,1)\times (-1, 0)$ and $\Omega_s := (0,1)\times (0, 0.5)$. We assume that the solid medium is isotropic and let the constitutive law \eqref{Constitutive:solid}  be  given in terms of the tensor 
\begin{equation}\label{eq:hooke} 
\cC_s \btau := 2 \mu_s\boldsymbol{\btau} + \lambda_s \tr(\boldsymbol{\btau}) I,
\end{equation}
where $\mu_s>0$ and $\lambda_s>0$ are the Lam\'e coefficients. The source terms $\bF_f$ and $\bF_s$ are chosen such that the manufactured solution \eqref{exactSol} satisfies \eqref{stressFormulation:solid}-\eqref{stressFormulation:fluid} with material parameters given by \eqref{L1} or \eqref{L2}. We observe that the boundary conditions $\bu_f = \mathbf{0}$ on $\Gamma_f$ and $\bu_s = \mathbf{0}$ on $\Gamma_s$ are satisfied, as prescribed in \eqref{stressFormulation:solid}-\eqref{stressFormulation:fluid}. Nevertheless, due to the discontinuity of the global stress tensor $\bsig$ across the interface $\Sigma$, the jump in normal components contributes to the right-hand side with the term
\[
\int_\Sigma (\bsig_f\bn_f + \bsig_s\bn_s)\cdot \hat\bv \, \text{d}s.
\]

In our first test, the parameters are chosen as 
\begin{gather}\label{L1}
	\rho_s = 1, \quad \mu_s = 1, \quad \lambda_s = 1, \quad \rho_f = 1, \quad \mu_f = 1, \quad \lambda_f = 10^6.
\end{gather}
\begin{table}[!htb]
	\centering
	\begin{minipage}{0.47\linewidth}
		\centering
{\footnotesize
\begin{tblr}{hline{1,22} = {1.5pt,solid}, hline{2, 7, 12, 17} = {1pt, solid},
	hline{6, 11, 16, 21} = {dashed}, vline{ 3} = {dashed}, colspec = {c l  c c c},}  
	 \textbf{$k$} & \textbf{$h$} & $\mathtt{e}^{L}_{hk}(\bsig)$ & $\mathtt{e}^{L}_{hk}(\bu)$ & $\mathtt{e}^{L}_{hk}(p)$ 
	 \\ 
	 \SetCell[r=4]{c} 0 
	& 1/8   & 2.26e+00   &   2.35e-01   &     3.09e+00 \\
	& 1/16  & 1.01e+00   &   6.12e-02   &     1.17e+00 \\
	& 1/32  & 4.92e-01   &   1.62e-02   &     5.33e-01 \\
	& 1/64  & 2.40e-01   &   5.10e-03   &     2.46e-01 \\
	 rates &   &    $1.08$     &   $1.84$  & $1.21$ \\ 
	 \SetCell[r=4]{c} 1 
	& 1/8    & 4.37e-01   &   1.01e-02   &     4.29e-01 \\ 
	& 1/16   & 9.51e-02   &   1.18e-03   &     8.93e-02 \\
	& 1/32   & 2.40e-02   &   1.58e-04   &     2.12e-02 \\
	& 1/64   & 5.89e-03   &   2.04e-05   &     4.94e-03 \\ 
	rates  &  & $2.07$ & $2.98$ & $2.14$ \\ 
	 \SetCell[r=4]{c} 2 
	& 1/8    & 7.70e-02   &   5.17e-04   &   4.40e-02 \\ 
	& 1/16   & 8.28e-03   &   2.94e-05   &   4.20e-03 \\
	& 1/32   & 1.08e-03   &   1.78e-06   &   5.19e-04 \\
	& 1/64   & 1.33e-04   &   1.23e-07   &   6.17e-05 \\ 
	rates  &  & $3.06$ & $3.97$ & $3.16$ \\ 
	 \SetCell[r=4]{c} 3 
	& 1/4   & 7.17e-02   &   4.66e-03   &     7.69e-02 \\ 
	& 1/8   & 3.42e-03   &   9.66e-05   &     4.32e-03 \\
	& 1/16  & 1.45e-04   &   2.00e-06   &     1.99e-04 \\
	& 1/32  & 9.11e-06   &   5.94e-08   &     1.29e-05 \\ 
	rates &  & $4.31$ & $5.41$ & $4.18$ \\
\end{tblr}
}	
\caption{Error progression and convergence rates are shown for a sequence of uniform refinements in space and over-refinements in time. The errors are measured at $T=0.3$, by employing the set of coefficients \eqref{L1}. The exact solution is given by \eqref{exactSol}.}
		\label{T1}
\end{minipage}\hfill
\begin{minipage}{0.47\linewidth}
	\centering
{\footnotesize 
\begin{tblr}{hline{1,22} = {1.5pt,solid}, hline{2, 7, 12, 17} = {1pt, solid},
	hline{6, 11, 16, 21} = {dashed}, vline{ 3} = {dashed}, colspec = {c l  c c c},} 
	 \textbf{$k$} & \textbf{$h$} & $\mathtt{e}^{L}_{hk}(\bsig)$ & $\mathtt{e}^{L}_{hk}(\bu)$ & $\mathtt{e}^{L}_{hk}(p)$ \\ 
	 \SetCell[r=4]{c} 1
	& 1/8   & 1.93e+03   &   2.63e+02   &     2.61e+03 \\ 
	& 1/16  & 4.20e+02   &   7.12e+01   &     6.21e+02 \\
	& 1/32  & 7.99e+01   &   1.16e+01   &     1.32e+02 \\
	& 1/64  & 1.66e+01   &   1.58e+00   &     2.53e+01 \\
	 rates &  & $2.28$ & $2.46$ & $2.22$ \\ 
	 \SetCell[r=4]{c} 2
	& 1/8    & 1.59e+02   &   2.98e+01   &     5.73e+02 \\ 
	& 1/16   & 1.74e+01   &   2.94e+00   &     5.13e+01 \\
	& 1/32   & 1.64e+00   &   1.86e-01   &     4.27e+00 \\
	& 1/64   & 1.85e-01   &   1.20e-02   &     3.64e-01 \\ 
	rates  &  & $3.24$ & $3.76$ & $3.54$ \\ 
	\SetCell[r=4]{c} 3
	& 1/4   & 3.32e+02   &   6.72e+01   &     6.78e+02 \\ 
	& 1/8   & 1.89e+01   &   2.97e+00   &     3.42e+01 \\
	& 1/16  & 9.43e-01   &   9.63e-02   &     1.77e+00 \\
	& 1/32  & 4.15e-02   &   2.67e-03   &     8.49e-02 \\
	rates  &  & $4.32$ & $4.87$ & $4.32$ \\ 
	\SetCell[r=4]{c} 4
	& 1/4   & 2.02e+01   &   7.87e+00   &     5.09e+01 \\ 
	& 1/8   & 6.76e-01   &   6.42e-02   &     1.71e+00 \\
	& 1/16  & 1.70e-02   &   1.11e-03   &     4.06e-02 \\
	& 1/32  & 3.83e-04   &   1.43e-05   &     8.39e-04 \\
	rates &  & $5.23$ & $6.35$ & $5.29$ \\ 
\end{tblr}
}
\caption{Error progression and convergence rates are shown for a sequence of uniform refinements in space and over-refinements in time. The errors are measured at $T=0.3$, by employing the set of coefficients \eqref{L2}. The exact solution is given by \eqref{exactSol}.}
\label{T2}
\end{minipage}
\end{table}
The time interval $[0, T]$ is equally divided into sub-intervals of length $\Delta t$. Since the error from the Crank-Nicolson is $O(\Delta t^2)$, we choose $\Delta t \simeq O(h^{(k+2)/2})$  so that the error from the time discretization does not pollute the order of convergence of the space discretization. For tables and figures presenting accuracy verification, we use the following notation for the $L^2-$norms of the errors:  
\begin{equation*}
\mathtt{e}^{L}_{hk}(\bsig) := \norm{\bsig(T) - \bsig_h^L}_{\cH},  \qquad \mathtt{e}^L_{hk}(\bu)  := \norm{\bu(T)  - \bu^L_h}_{0,\Omega}.
\end{equation*}
 The rates of convergence in space are computed as 
\begin{equation}\label{rate}
	\mathtt{r}_{hk}^L(\star)  =\log(\mathtt{e}^L_{hk}(\star)/\tilde{\mathtt{e}}^L_{hk}(\star))[\log(h/\tilde{h})]^{-1}
\quad \star \in \set{\bsig, \bu, p},
\end{equation}
where $\mathtt{e}_{hk}^L(\star)$, $\tilde{\mathtt{e}}^L_{hk}(\star)$ denote errors generated at time $T$ on two consecutive  meshes of sizes $h$ and~$\tilde{h}$, respectively.  

In Table~\ref{T1} we present  the errors, at the final time $T = 0.3$ relative to the mesh size $h$ for four  different polynomial degrees $k$. We also include the arithmetic mean of the experimental rates of convergence obtained via \eqref{rate}. We observe that the convergence in the stress field achieves the optimal rate of $O(h^{k+1})$. Furthermore, Table~\ref{T1} highlights a superconvergence rate of $O(h^{k+2})$ for the velocity. We point out that, the pressure field is postprocessed through the formula
\[
p = -\frac{\lambda_F}{2\mu_F + d \lambda_F} \tr \bsig.
\] 

\begin{table}[h!]
	\begin{minipage}{0.47\linewidth}
	  \centering
	  {\footnotesize
	  \begin{tblr}{hline{1,7} = {1.5pt,solid},
		  hline{2} = {dashed},
		  vline{2} = {dashed},
		  colspec = {lX[c]X[c]X[c]X[c]},}
		  $\Delta t$ & $\mathtt{e}^{L}_{hk}(\bsig)$ & $\mathtt{r}_{hk}^L(\bsig)$ & $\mathtt{e}_{hk}^L(\bu)$ & $\mathtt{r}_{hk}^L(\bu)$ \\
		  1/16 & 3.14e-03 & * & 1.63e-04 & * \\
		  1/32 & 8.10e-04  & 1.96 &  4.19e-05 &  1.96 \\
		  1/64 & 2.05e-04  & 1.98  & 1.06e-05 &  1.99   \\
		  1/128 & 5.16e-05 &  1.99 &  2.67e-06  & 1.98 \\
		  1/256 & 1.29e-05 &  1.99 &  6.85e-07  & 1.96 \\
	  \end{tblr}
	  }
	  \captionof{table}{Computed errors for a sequence of uniform refinements in time with $h=1/20$ and $k=5$. The errors are measured at $t=T=1$, with the coefficients \eqref{L1}. The exact solution is provided by \eqref{exactSol}.}
	  \label{T3}
	\end{minipage}
	\hfill
	\begin{minipage}{0.47\linewidth}
	  \centering
	  \includegraphics[width=0.8\textwidth]{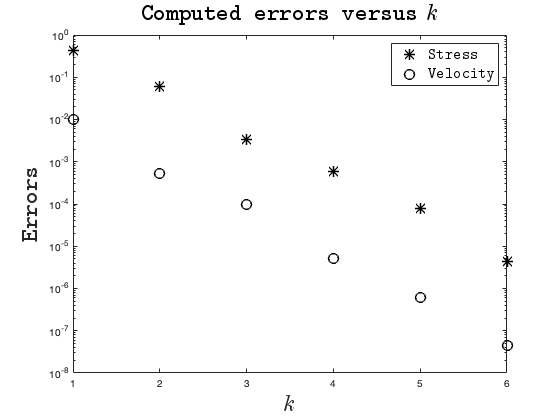}
	  \captionof{figure}{Computed errors versus the polynomial degree $k$ with $h=1/8$ and $\Delta t = 10^{-4}$. The errors are measured at $t=T=0.3$, by employing the coefficients \eqref{L1}. The exact solution is provided by \eqref{exactSol}.}
	  \label{T4}
	\end{minipage}
  \end{table} 

To verify the accuracy and stability of the scheme in the nearly incompressible limit we repeat the same experiment with the set of material parameters
\begin{gather}\label{L2}
	\rho_s = 10^3, \quad \mu_s = 10^6, \quad \lambda_s = 10^{10}, \quad \rho_f = 1, \quad \mu_f = 1, \quad \lambda_f = 10^6.
\end{gather}
The error decay for this case is collected in Table~\ref{T2}.  These results demonstrate the ability of the proposed HDG scheme to produce accurate approximations also in the nearly incompressible elasticity regime.

On the other hand, Table~\ref{T3} portrays the convergence results obtained after fixing the mesh size to $h=1/16$ and the polynomial degree to $k= 5$ and  varying the time step $\Delta t$ used to subdivide the time interval $[0,T]$ uniformly, with $T=0.3$. The rates of convergence in time, are computed as
\[
\mathtt{r}_{hk}^L(\star)  =\log(\mathtt{e}^L_{hk}(\star)/\tilde{\mathtt{e}}^L_{hk}(\star))[\log(\Delta t/\widetilde{\Delta t})]^{-1}\quad \star \in \set{\bsig, \bu, p},
\]
where $\mathtt{e}^L_{hk}$, $\tilde{\mathtt{e}}^L_{hk}$ denote errors generated on two consecutive runs considering time steps $\Delta t$ and~$\widetilde{\Delta t}$, respectively. 
In this example, we consider the same manufactured solution obtained from \eqref{exactSol} with material coefficients \eqref{L1}. The expected convergence rate of $O(\Delta t^2)$ is attained as the time step is refined.

Finally, we fix the space mesh size $h = 0.25$ and the time mesh size $\Delta t = 10^{-6}$ and let $k$ vary from 1 to 6. In Figure~\ref{T4} we report the error $\mathtt{e}^{L}_{hk}(\bsig)$ in the stress variable and the error $\mathtt{e}^{L}_{hk}(\bu)$ in velocity at  $T = 0.3$  as a function of the polynomial degree $k$ in a semi-logarithmic scale. As expected, an exponential convergence is observed.

\subsection{Example 2: Linear hemodynamics problem in two dimensions}

We consider a simplified hemodynamics problem commonly employed as a benchmark to validate linear fluid-structure Interaction (FSI) solvers \cite{bukac2014,bukac2021, Fu2022}. It involves fluid flow in a two--dimensional channel interacting with a thick elastic wall. The fluid and structure domains are defined as $\Omega_f = (0,6)\times (0, 0.5)$ \unit{cm^2} and $\Omega_s = (0, 6)\times (0.5, 0.6)$ \unit{cm^2}, respectively. 

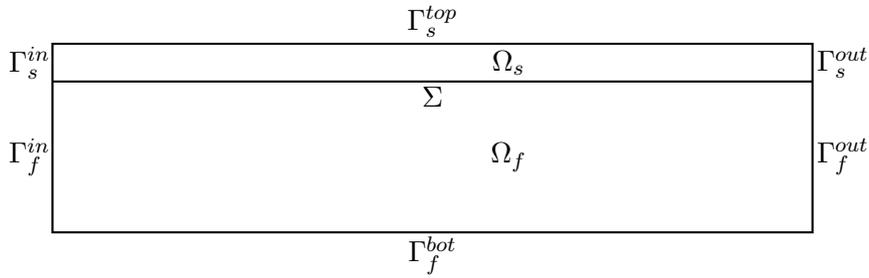
\begin{figure}[h]
	\centering
\begin{tikzpicture}
    
    \draw [thick] (0,0) -- (10,0) -- (10,2.5) -- (0,2.5) -- cycle;
    \draw [thick] (0,2) -- (10,2);
    
    \node at (5,1.8) {$\Sigma$};
	\node at (-0.3,1) {$\Gamma_{f}^{in}$};
	\node at (10.4,1) {$\Gamma_{f}^{out}$};
	\node at (-0.3,2.25) {$\Gamma_{s}^{in}$};
	\node at (10.4,2.25) {$\Gamma_{s}^{out}$};
	\node at (5,-0.3) {$\Gamma_{f}^{bot}$};
	\node at (5,2.8) {$\Gamma_{s}^{top}$};

	\node at (6,1) {$\Omega_{f}$};
	\node at (6,2.25) {$\Omega_{s}$};

\end{tikzpicture}
\caption{Fluid and solid domains.}
\label{FIG1}
\end{figure}

We consider the fluid-structure interaction problem \eqref{stressFormulation:solid}--\eqref{stressFormulation:fluid}, where the source terms $\bF_s$ and $\bF_f$ are assumed to vanish. Additionally, we introduce a modification to the elastodynamic equation within the solid domain as follows: 
\[
	\rho_s\dot{\bu}_s - \bdiv \bsig_s + \beta_s \bd = \mathbf 0  \quad\text{in $\Omega_s\times (0, T]$}.
\]
Here, the term $\beta_s \bd$ represents the linearly elastic spring component, which is used to take into account the radial symmetry of the artery in two dimensions, as detailed in \cite{bukac2014}. We highlight that the inclusion of $\beta_s \bd$ necessitates treating the elastic displacement as a further unknown in our formulation, which requires the introduction of the additional equation $\partial_t \bd = \bu_s$ within the solid domain. We approximate the elastic displacement variable $\bd$ using the same polynomial space employed for the velocity field $\bu_s$, thereby introducing only further locally coupled degrees of freedom.

The values of all the fluid and structure parameters in this example are given by 
\begin{align}\label{sang}
	\begin{split}
		\rho_s &= 1.1\, \unit{g\per cm^3} , \quad \mu_s = 0.575\times 10^6\, \unit{dyne\per cm^2} , \quad \lambda_s = 1.7\times 10^6\, \unit{dyne \per cm^2} 
		\\
		 \rho_f & = 1\, \unit{g \per cm^3} , \quad \mu_f = 1\, \unit{g \per cm.s} , \quad \beta_s = 4\times 10^6\, \unit{dyne \per cm^4}.	
	\end{split}
\end{align}

\begin{figure}[h!]
	\begin{center}
		\includegraphics[width=0.9\textwidth]{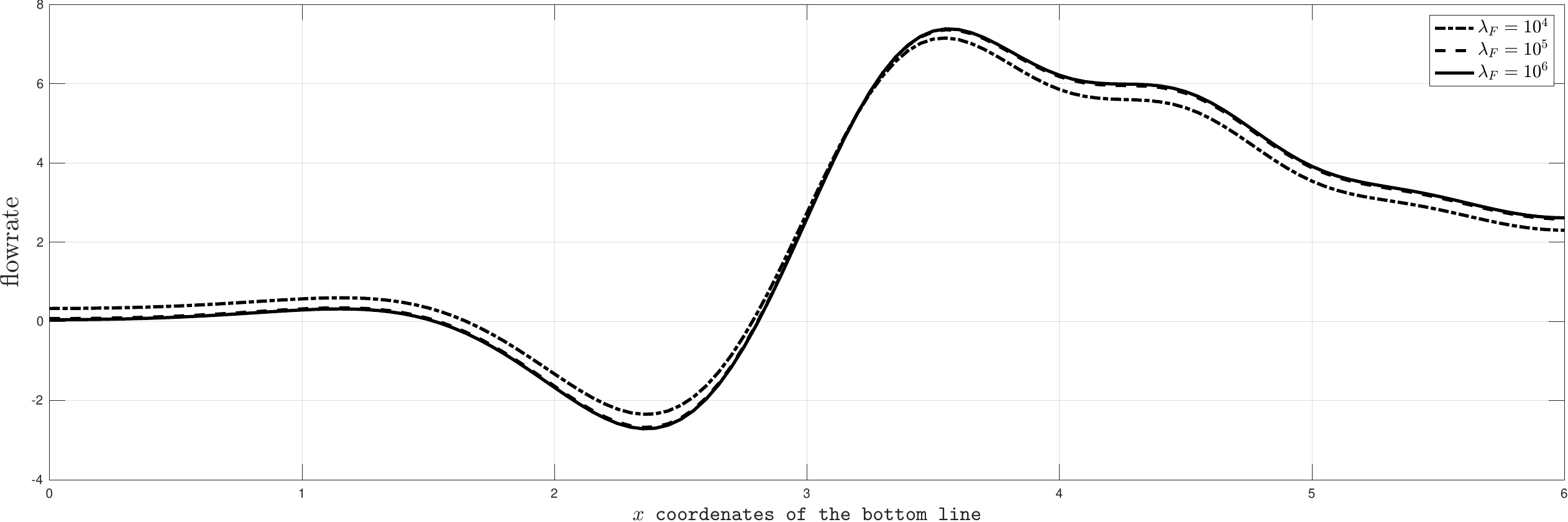}
		\includegraphics[width=0.9\textwidth]{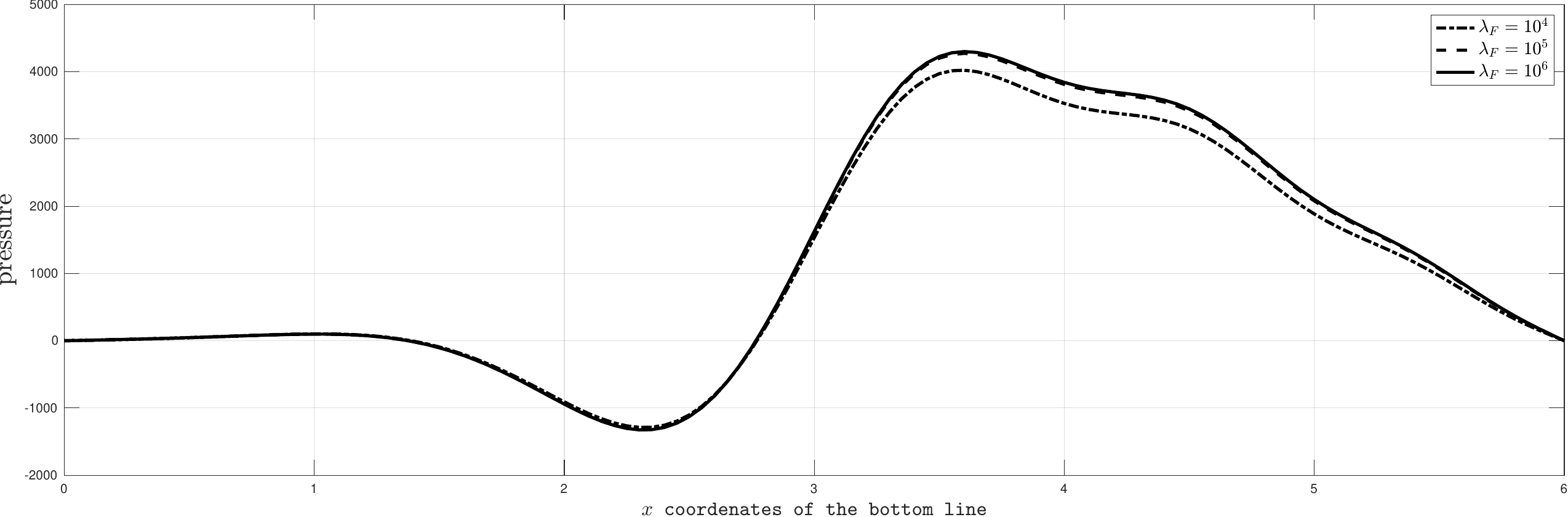}
		\includegraphics[width=0.9\textwidth]{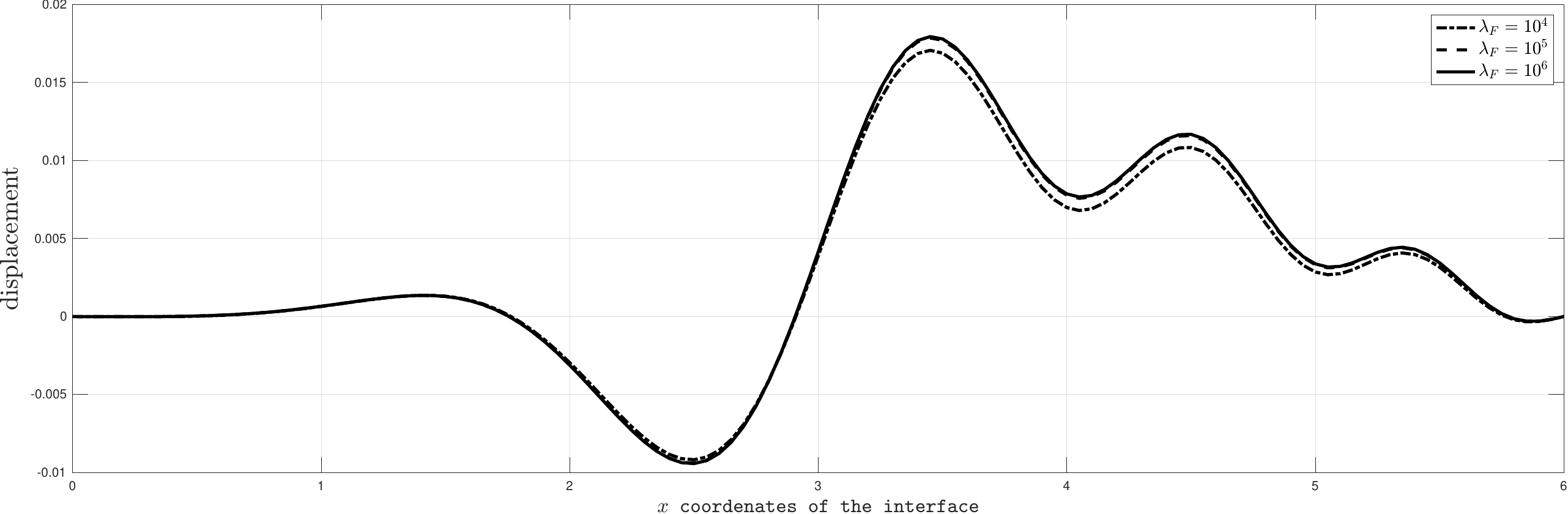}
	\end{center}
	\vspace{-4mm}
	\caption{Numerical solutions of Example 2 at final time $T = 0.012$ \unit{s}, obtained with the discretization parameters $h= 0.1$, $k=2$, $\Delta t = 10^{-4}$, and  different incompressibility parameters $\lambda_f$. Top: $x$--component of  $\frac23 \bu_{f,h}$ along bottom line $\Gamma^{bot}_f$; Middle: pressure along bottom line $\Gamma^{bot}_f$; Bottom: $y$--component of $\bd$ along the interface $\Sigma$.}\label{fig:disp}
\end{figure}

We set zero initial conditions in \eqref{inits} and \eqref{initf} and assume that the flow is driven by a prescribed time-dependent pressure $p_{in}(t)$  at the inlet boundary $x=0$ defined as
\begin{equation}\label{pulse}
	p_{in}(t) = \begin{cases}
		\frac{p_{max}}{2} \left(1 - \cos(\frac{2\pi t}{t_{max}})\right) & \text{if $t\leq t_{max}$}
		\\
		0 & \text{if $t > t_{max}$},
		\end{cases}	
\end{equation}
where $t_{max} = 0.003$ \unit{s} and $p_{max} = 1.333 \times 10^4$ \unit{dyne \per cm^2}. Furthermore, zero pressure is enforced at $x=6$, while a symmetry condition is applied to the lower boundary $y=0$. The structure is clamped at $x=0$ and $x=6$, and zero traction is applied at $y=0.6$. The specific boundary conditions are detailed as follows: (See Figure~\ref{FIG1} for an illustration of the domains and boundary labels.)
\begin{align}\label{bcs}
	\begin{aligned}
		(\bsig\bn_f)\cdot \bn_f &= - p_{in}(t), &  \text{tang}(\bu) &=  \mathbf 0 \quad \text{on $\Gamma_f^{in} = \set{0}\times (0, 0.5)$},
		\\
		(\bsig\bn_f)\cdot \bn_f &= 0, &   \text{tang}(\bu) &= \mathbf 0 \quad \text{on $\Gamma_f^{out} = \set{6}\times (0, 0.5)$},
		\\
		(\bsig\bn_s)\cdot \bn_s &= 0, &  \text{tang}(\bu) &= \mathbf 0 \quad \text{on $\Gamma_{s}^{top} = (0, 6) \times \set{0.6}$},
		\\
		\text{tang}(\bsig\bn_s) &= \mathbf 0,  & \bu\cdot \bn_s &= 0 \quad \text{on $\Gamma_{s}^{bot} = (0, 6) \times \set{0}$},
		\\
		\bu &= \mathbf 0  & &  \qquad \quad \text{on $\Gamma_{s}^{in}\cup \Gamma_{s}^{out} = \set{0} \times (0.5, 0.6) \cup \set{6} \times (0.5, 0.6)$}. 
	\end{aligned}
\end{align}
Here, we use the notation $\text{tang}(\bv) = \bv - (\bv\cdot \bn)\bn$ to denote the tangential component of a vector field on a boundary with unit normal vector $\bn$.

We solve the problem with a mesh size $h= 0.1$, a polynomial degree $k=2$, and a uniform time step $\Delta t = 10^{-4}$. To validate the penalty strategy adopted in the formulation of the Stokes problem, we compare the solutions obtained at time $T = 0.012$ \unit{s} for incompressibility parameters $\lambda_f \in \set{ 10^4, 10^5, 10^6 }$. For each case, we plot in Figure~\ref{fig:disp} the flow rate, calculated as two thirds of the horizontal velocity, following \cite{Fu2022}), and the pressure at the bottom boundary $\Gamma_f^{bot}$, along with the vertical displacement on the interface $\Sigma$. Notably, for the two largest values of $\lambda_f$, the plotted curves exhibit negligible distinctions. Furthermore, our results show good agreement with those presented in \cite[Figure 2]{Fu2022}.

\begin{figure}[h!]
	\begin{center}
	\includegraphics[width=0.6\textwidth]{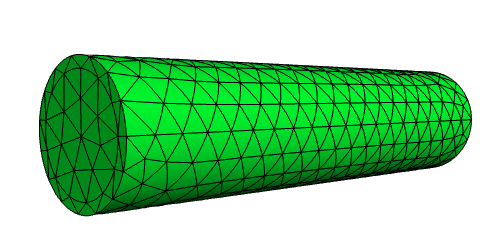}
	\end{center}
	\vspace{-4mm}
	\caption{The computational domain partitioned with a mesh size $h = 0.25$.}\label{mesh}
\end{figure}

\begin{figure}[h!]
	\begin{center}
		\includegraphics[width=1\textwidth]{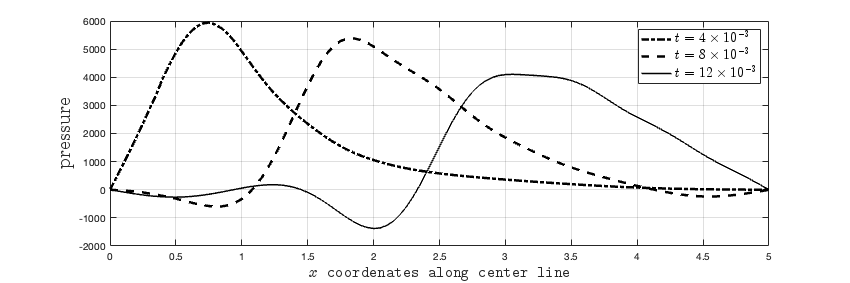}
		\includegraphics[width=1\textwidth]{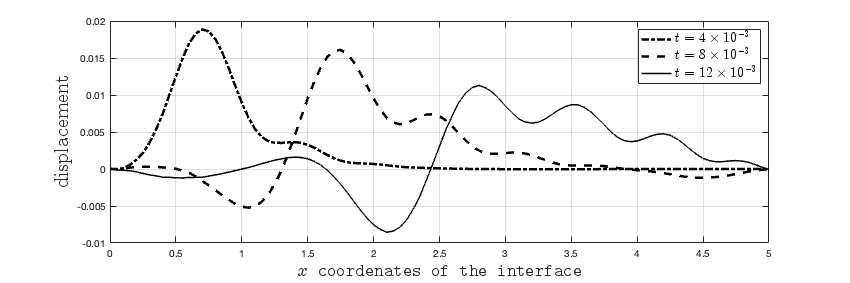}
	\end{center}
	\vspace{-4mm}
	\caption{Numerical solutions of Example 3  at progressively increasing final times. Discretization parameters used: mesh size $h= 0.25$, polynomial degree $k=3$, and time step $\Delta t = 10^{-4}$. Top: Pressure distribution along the center line $\set{(x,0,0)\ 0 < x < 5}$. Bottom: $y$--component of the displacement on the  line $\set{(x, 0.55, 0) : \ 0 <x<5 }$.}\label{fig:disp3D}
\end{figure}

\begin{figure}[h!]
	\begin{center}
	\includegraphics[width=0.7\textwidth]{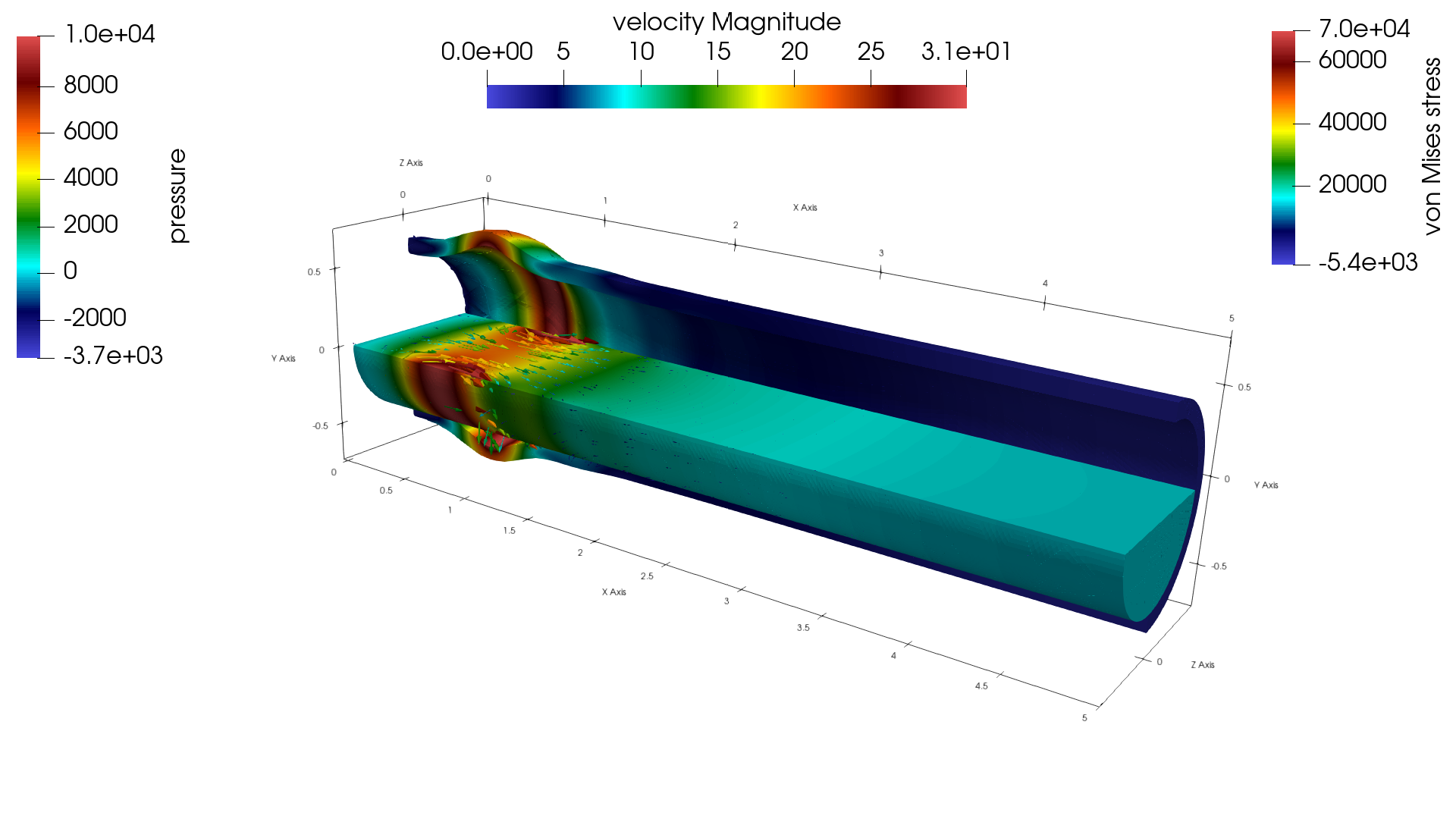}
	\includegraphics[width=0.7\textwidth]{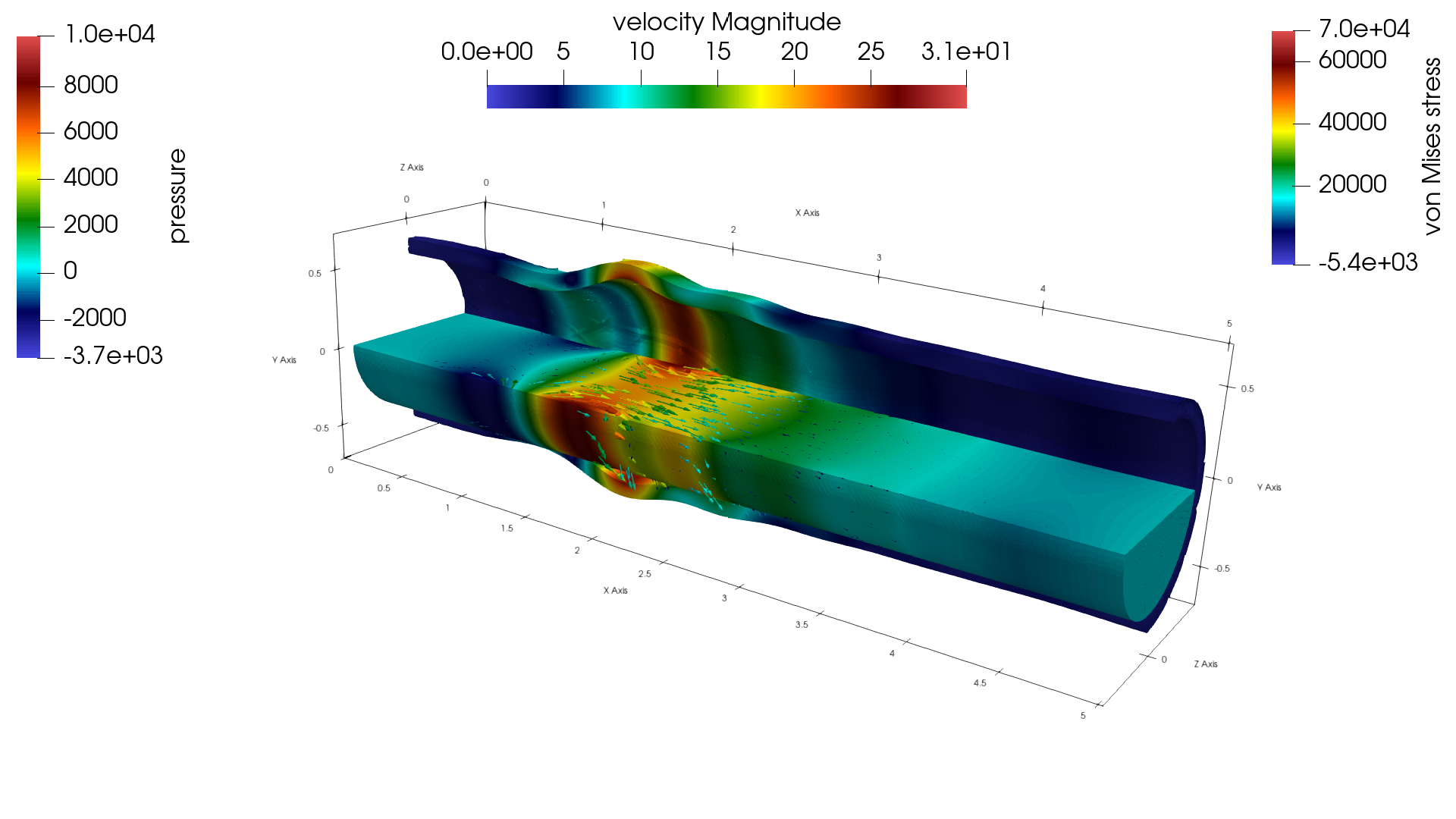}
	\includegraphics[width=0.7\textwidth]{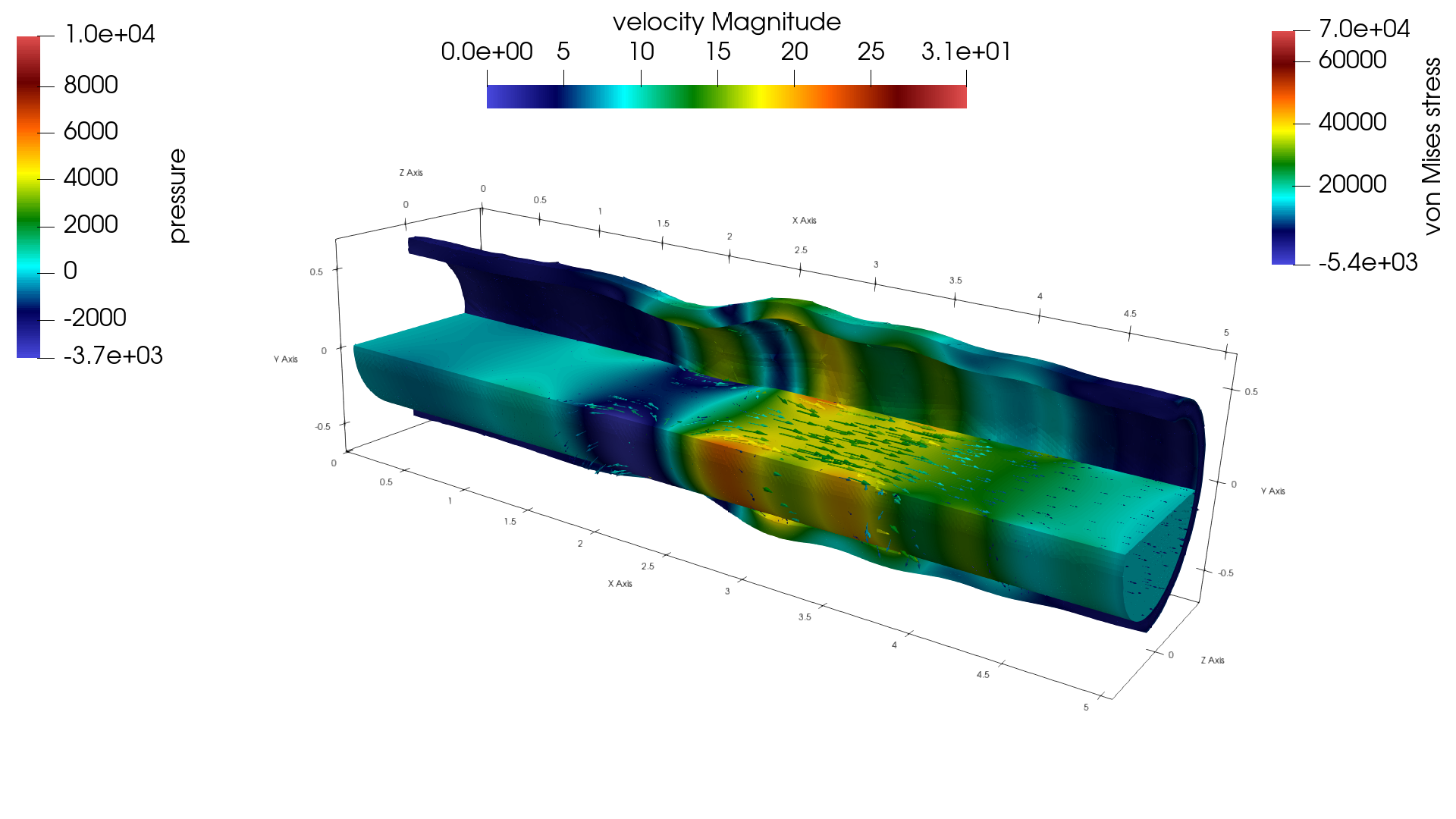}
	\end{center}
	\vspace{-6mm}
	\caption{The figure presents snapshots of the numerical solution for Example 3 at three distinct times: $t = 4\times 10^{-3}$ \unit{s}, $t = 8\times 10^{-3}$ \unit{s},  and $t = 12\times 10^{-3}$ \unit{s}. The solid domain depicts the deformed structure, with the deformation magnified by a factor of 8 for better visualization. The color map within this domain represents the distribution of von Mises stresses. In contrast, the fluid domain utilizes a color map to visualize the pressure distribution. Additionally, the velocity field within the fluid domain is shown using arrows. The results are obtained using a mesh size of $h= 0.25$, a polynomial degree $k=3$, and a time step $\Delta t = 10^{-4}$.}\label{fig:tube3D}
\end{figure}

\subsection{Example 3: Linear hemodynamics problem in three dimensions}

We conclude the numerical experiments by revisiting the problem introduced in Example 2, extending it to three dimensions. We simulate the propagation of a pressure pulse within a straight cylinder with a length of 5 \unit{cm}. The fluid domain $\Omega_f = \set{(x,y,z):\ x\in (0,5), y^2 + z^2 < 0.5^2}$ corresponds to the interior of the cylinder, while the structural part $\Omega_s = \set{(x,y,z):\ x\in (0,5), 0.5^2 < y^2 + z^2 < 0.6^2}$ consists in a tubular region with a thickness of 0.1 \unit{cm}. Consequently, the interface between the fluid and the structure is  given by $\Sigma = \set{(x,y,z):\ x\in (0,5), y^2 + z^2 = 0.5^2}$. 

We choose an incompressibility coefficient $\lambda_f = 10^{6}$ and employ the material parameters defined in equation~\eqref{sang}, excluding the spring component by setting $\beta_s = 0$. The boundary conditions are once again those specified in equation~\eqref{bcs}, with one exception.  A pure Neumann boundary condition, $\bsig \bn = \mathbf 0$ is applied on the top boundary structure  $\Gamma_s^{top} = \set{(x,y,z):\ x\in (0,5),  y^2 + z^2 = 0.6^2}$. This effectively clamps the structure at both ends, while the fluid is driven by the pressure pulse $p_{in}(t)$, defined in \eqref{pulse}, at $\Gamma_f^{in} = \set{(x,y,z):\ x=0,\ y^2 + z^2 = 0.5^2}$. 

We present a numerical solution to problem~\eqref{fd} utilizing the aforementioned initial and boundary conditions. The discretization employs a mesh size $h= 0.25$, a polynomial degree $k=3$, and a uniform time step  $\Delta t = 10^{-4}$. 

Figure~\ref{fig:disp3D} visualizes the pressure distribution along the centerline $\set{(x, 0, 0) :\ 0 <x<5}$, and the $y$-component of the displacement on the  line $\set{(x, 0.55, 0) : \ 0 <x<5 }$ at final times $T \in \set{4 \times 10^{-3}, 8\times 10^{-3}, 12 \times 10^{-3}}$. The approximate displacements within the structure are obtained through post-processing of the velocity field using the formula
\[
  \bd_h^{n+1} = \bd_h^{n} + \Delta t \ (\bu^{n+1}_{s,h} + \bu^{n}_{s,h})/2\quad \text{in $\Omega_s$}.
\]

Figure~\ref{fig:tube3D} further illustrates the propagation of the pressure pulse over time. This figure depicts the deformation and von Mises stress within the structure, alongside the pressure and velocity fields in the fluid domain.

\section{Conclusion}
We have developed a novel velocity/stress-based formulation for a linear Fluid-Structure Interaction problem involving a thick structure. We have established the well-posedness and energy stability of the variational formulation. Additionally, we have introduced a hybridizable discontinuous Galerkin  discretization method  tailored to this problem, and performed an $hp$-convergence analysis of the  semi-discrete scheme. Moreover, we have employed the Crank-Nicolson method for temporal discretization, and examined the convergence properties of the resulting fully-discrete problem. The generalization to more realistic FSI models is under investigation.

\bibliographystyle{plainnat}
\bibliography{cvBib.bib}

\end{document}